\documentclass[10pt]{article}

\usepackage{calc}

\usepackage{amsmath}
\usepackage{amsxtra}
\usepackage{amssymb}
\usepackage{amsfonts}
\usepackage{amsthm}
\usepackage{amstext}
\usepackage{amsbsy}
\usepackage{amscd}
\usepackage[sans]{dsfont}

\usepackage{color}
\usepackage{bbding}

\newcommand{\ucon}[1]{(\textcolor{red}{TODO:} #1)}

\newcommand{\R}{\mathds{R}}

\newcommand{\sgn}{\operatorname{sgn}}

\newcommand{\loc}{{\operatorname{loc}}}

\theoremstyle{plain}
\newtheorem{theorem}{Theorem}
\newtheorem{lemma}{Lemma}

\theoremstyle{definition}

\theoremstyle{remark}

\def\XXint#1#2#3{{\setbox0=\hbox{$#1{#2#3}{\int}$}
\vcenter{\hbox{$#2#3$}}\kern-.5\wd0}}

\newcommand{\aint}{{-\kern-5.3mm\int}}

\renewcommand{\ucon}[1]{\oops}
\usepackage{graphicx}

\usepackage{bbding}
\usepackage{perpage}
\usepackage{mathtools}

\newcommand{\myref}[1]{\ref{#1}}
\newcommand{\mylabel}[1]{\label{#1}}
\newcommand{\myeqlabel}[1]{\label{#1}}

\usepackage{calc}
\usepackage{enumerate}
\usepackage{amsmath}
\usepackage{amsxtra}
\usepackage{amssymb}
\usepackage{amsfonts}
\usepackage{amsthm}

\usepackage[sans]{dsfont}

\usepackage{color}
\definecolor{lightgrey}{rgb}{0.7,0.7,0.7}

\usepackage{textcomp}
\usepackage{mathrsfs}

\newcommand{\Ub}{\overline{U}}

\newcommand{\bO}{\mathcal{O}}

\theoremstyle{plain}

\title{Steady and self similar full Euler flow\footnote{This material is based upon work supported by the
National Science Foundation under Grant No.\ NSF DMS-0907074.}}
\author{Joseph Roberts}

\begin{document}

\maketitle

\begin{abstract}
We consider solutions to the full (non-isentropic) two-dimensional Euler equations that are constant in time and along rays emanating from the origin.  We prove that for a polytropic equation of state, entropy admissible solutions in $L^\infty$ with non-vanishing velocity, density, and internal energy must be $BV$.  Moreover, we obtain some results concerning the structure of such solutions.
\end{abstract}

\section{Introduction}
The study of multi-dimensional compressible gas flow (and more generally, multi-dimensional hyperbolic conservation laws) is a very active field of research, and there are many open questions.  As a modest step toward understanding multi-dimensional flow, several reductions that are suggested by experiments can be made.  In flow patterns such as regular and Mach reflection, the flow can be simplified to be constant in time and along rays emanating from a distinguished point.  A regular reflection consists of four shock waves meeting at a point, (see \cite{chen-feldman-selfsim-journal,elling-liu-pmeyer,canic-keyfitz-lieberman,yuxi-zheng-rref,henderson-etal,elling-rrefl-lax,elling-sonic-potf,elling-detachment}), while a Mach reflection consists of three shocks and a contact discontinuity meeting at a point (see \cite{ben-dor-book,ben-dor-shockwaves2006,hornung-reviews,hunter-tesdall,vasiliev-kraiko,skews-1997}).  Therefore, we consider solutions $U(t,x,y)$ that are in fact of the form $U(\theta)$, where $\theta = \angle(x,y)$.  Our goal is to investigate the regularity and structure of such solutions. \\

Historically, understanding the \emph{Riemann problem} was instrumental in the development of the theory of one dimensional conservation laws.  In that case, one seeks solutions to the Cauchy problem in which the initial data is constant along rays emanating from the origin, which in one space dimension is of the form
\begin{alignat}{1}
U(0,x) = \left\{ \begin{array}{ll} U_L & x < 0 \\ U_R & x > 0 \end{array} . \right  .
\end{alignat}
Riemann problems are the foundation of theoretical existence and uniqueness proofs for general Cauchy problems (the Glimm scheme \cite{glimm} being the main example), as well as of numerical schemes that accurately capture shocks (such as the Gudonov and Roe methods).  \\

The two-dimensional Riemann problem is much more complicated --- some work has been done in which the initial data is constant in each quadrant (see \cite{li-zhang-yang,zhang:593,yuxi-zheng-book,lax-xdliu}).  Instead of focusing on that kind of initial value problem, we seek steady in time solutions with more general configurations. \\

Solutions of this form are also motivated by questions of non-uniqueness.  In \cite{elling-nuq-journal}, Elling numerically constructed an unsteady solution in which the initial data is a steady solution that is constant along rays.  Therefore, understanding this class of solutions may lead to more examples of non-uniqueness, perhaps even analytical proofs.\\

Another interesting issue for multi-dimensional conservation laws is finding the best function space to utilize.  Whereas $BV$ is the natural choice for one spatial dimension (see \cite{glimm,glimm-lax,bianchini-bressan}), it is well known that it is not suitable for the multi-dimensional case (see \cite{rauch-bv}).  However, when restricted to steady and self-similar solutions, the two-dimensional Euler equations are similar to one-dimensional self-similar conservation laws, and so perhaps $BV$ is an appropriate choice in this setting.  This is one of the major results of this work --- that under reasonable assumptions steady and self-similar Euler flows that are only assumed to be $L^\infty$ can be shown to be $BV$.  \\

This paper is an extension of the work in \cite{ER}, but there are several important differences.  In that case, the solutions were assumed to be small perturbations of a constant supersonic state, and so the steady problem was still strictly hyperbolic.  In this case, we permit large variations in the solution.  When the flow is sonic, there is a parabolic degeneracy in the steady problem; for subsonic flow the system becomes hyperbolic-elliptic.  Therefore, a different approach from the one in \cite{ER} must be utilized. \\

Moreover, the analysis in \cite{ER} required that the eigenvalues be distinct --- in full Euler flow there is a double eigenvalue corresponding to shear waves and entropy jumps.  Fortunately, many of the difficulties in dealing with this degeneracy can be avoided since when dealing with the Euler system we have explicit expressions for the shock curves and averaged matrices, and do not need to rely on the implicit function theorem.\\

Sections 1-6 detail the various assumptions, derive the appropriate weak form of the equations, recall important properties of shock transitions, and demonstrate how to work with the solutions in the $L^\infty$ setting.  Section 7 shows that there can be a finite number of contact discontinuities, proves that shocks have neighborhoods on either side on which the flow is constant, and demonstrates what can go wrong if the velocity vanishes.  All that is required for these results is a standard equation of state.  Sections 8 and 9 go into more detail about the configuration of shocks and prove $BV$ regularity for the case of a polytropic gas.  Finally, sections 10 and 11 describe results regarding the structure of possible solutions in the polytropic case --- how many waves can there be, what configurations can occur, etc.  

\section{Preliminaries}
Let $\mathcal{P} \in \R^m$ be an open set.  Consider smooth functions $\eta, \psi^x, \psi^y : \mathcal{P} \rightarrow \R.$  For $A \subset \R^3$  we say $U = (U^1, ..., U^m) \in L^1_{\loc}(\R^3, \mathcal{P})$ is a \textit{weak solution} of
\begin{alignat}{1}
\eta(U)_t + \psi^x(U)_x + \psi^y(U)_y \geq 0 \qquad \mbox{ in } A, \myeqlabel{weakent}
\end{alignat}
if the inequality is satisfied in the weak sense: every $x \in A$ has an open neighborhood $N$ so that for nonnegative smooth $\Phi$ with supp $\Phi \Subset N$,
\begin{alignat}{1}
- \int_{\R^3} \Phi_t \eta(U) + \Phi_x \psi^x(U) + \Phi_y \psi^y(U) d(x,y,t) \geq 0. \myeqlabel{testent}
\end{alignat}
We call $U$ a \textit{strong solution} (or \textit{classical solution}) if, in addition, it is almost everywhere equal to a Lipschitz continuous function.  Weak solutions for a system of conservation laws
\begin{alignat}{1}
U_t + f^x(U)_x + f^y(U)_y = 0 \qquad \mbox{ in } A, \myeqlabel{syst}
\end{alignat}
with $f^x, f^y : \mathcal{P} \rightarrow \R^m$ smooth, are defined by interpreting it as $2m$ inequalities with $=$ replaced by $\leq$ or $\geq$, and with $\eta(U) := U^\alpha, \psi^x(U) = f^{x\alpha}(U), \psi^y(U) := f^{y\alpha}(U)$ for $\alpha = 1,...m$.  
We call $(\eta, \psi^x, \psi^y)$ an entropy-entropy flux pair for \eqref{syst} if
\begin{alignat}{1}
\psi^x_U = \eta_U f^x_U, \quad \psi^y_U = \eta_U f^y_U \quad \mbox{ on } \mathcal{P}.
\end{alignat}
A weak solution $U$ of the Euler equations \eqref{syst} is called an entropy solution if it satisfies \eqref{weakent} for the physical concave entropy $\eta$, which will sometimes be denoted by $\rho S$.  The Euler equations are given by
\begin{alignat}{1}
U = \left( \begin{array}{c} \rho \\ \rho u \\ \rho v \\ \rho E \end{array} \right), f^x(U) = \left( \begin{array}{c} \rho u \\ \rho u^2 + p \\ \rho u v \\ u( E + p) \end{array} \right), f^y(U) = \left( \begin{array}{c} \rho v \\ \rho u v \\ \rho v^2 + p \\ v( E + p) \end{array} \right),
\end{alignat}
where $\rho$ is the density, $u$ and $v$ are the horizontal and vertical velocities, respectively, $p$ is the pressure, and $ E := \rho(\frac{u^2+v^2}{2} + e)$ is the total energy per unit volume, with $e$ the internal energy per unit mass.  \\
We now derive the weak form for steady, self-similar solutions $U$ which only depend on $\theta$, the polar angle $\angle(x,y)$.  Integrating by parts in $t$ eliminates the first term in the integrand of \eqref{weakent}, after which using compact-in-$t$ support and integrating with respect to $t$ yields the equivalent statement
\begin{alignat}{1}
- \int_{\R^2} \Phi_x \psi^x(U) + \Phi_y \psi^y(U) d(x,y) \geq 0
\end{alignat}  
for all nonnegative smooth compactly supported $t$-independent functions $\Phi : \R^2 \rightarrow \R$.  Now, change variables to polar coordinates to obtain
\begin{alignat}{1}
0 \leq - \int_0^\infty \int_0^{2 \pi} \Big(\Phi_x(r \cos \theta, r \sin \theta) \psi^x(U(\theta)) + \Phi_y(r \cos \theta, r \sin \theta) \psi^y(U(\theta))\Big) r d\theta dr. \myeqlabel{thetaweak}
\end{alignat}
Define a smooth $2\pi$ periodic function $\phi$ as
\begin{alignat}{1}
\phi(\theta) := \int_0^\infty \Phi(r \cos \theta, r \sin \theta) dr = - \cos \theta \int_0^\infty \Phi_x r dr - \sin \theta \int_0^\infty \Phi_y r dr.
\end{alignat}
Notice that
\begin{alignat}{1}
\phi'(\theta) = - \sin \theta \int_0^\infty \Phi_x r dr + \cos \theta \int_0^\infty \Phi_y r dr.
\end{alignat}
We then have that
\begin{alignat}{1}
- \int_0^\infty \Phi_x r dr &= \cos \theta \phi(\theta) + \sin \theta \phi'(\theta), \\
-\int_0^\infty \Phi_y r dr &= \sin \theta \phi(\theta) - \cos \theta \phi'(\theta).
\end{alignat}
Therefore \eqref{thetaweak} becomes
\begin{alignat}{1} 
0 \leq \int_0^{2 \pi} \Big(\cos \theta \phi + \sin \theta \phi' \Big) \psi^x(U(\theta)) + \Big( \sin \theta \phi - \cos \theta \phi'\Big) \psi^y(U(\theta)) d \theta.
\end{alignat}
This is the weak form of
\begin{alignat}{1}
\Big( - \sin \theta \psi^x(U(\theta)) + \cos \theta \psi^y(U(\theta)) \Big)_\theta + \cos \theta \psi^x(U(\theta)) + \sin \theta \psi^y(U(\theta)) \geq 0,
\end{alignat}
or equivalently, 
\begin{alignat}{1}
\Big( \sin \theta \psi^x(U(\theta)) - \cos \theta \psi^y(U(\theta)) \Big)_\theta - \cos \theta \psi^x(U(\theta)) - \sin \theta \psi^y(U(\theta)) \leq 0.
\end{alignat}
The left side is a nonpositive distribution, and is therefore a nonpositive Radon measure.  By the Riesz representation theorem (see for example \cite{Benedetto}, Chapter 7), it is the distributional derivative of a non-increasing right continuous function of bounded variation.  Therefore, there is a version (that is, an element of the same $L^\infty$ equivalence class)  of $U$ such that
\begin{alignat}{1}
\Big(\sin \theta \psi^x(U(\theta)) - \cos \theta \psi^y(U(\theta)) \Big) \Big|_{\theta_1}^{\theta_2} \leq \int_{\theta_1}^{\theta_2} \cos \eta \psi^x(U(\eta)) + \sin \eta \psi^y(U(\eta)) d\eta.
\end{alignat}
Interpreting the conservation law as $2m$ inequalities, we obtain
\begin{alignat}{1}
\Big(\sin \theta f^x(U(\theta))-\cos \theta f^y(U(\theta)) \Big)_\theta = \cos \theta f^x(U(\theta))+\sin \theta f^y(U(\theta)). \myeqlabel{finalweak}
\end{alignat}
The right side is $L^\infty$, and so the quantity being differentiated in the distributional sense on the left must have a version that is Lipschitz.  Therefore, the fundamental theorem of calculus holds and we have
\begin{alignat}{1}
\Big(\sin \theta f^x(U(\theta))- \cos \theta f^y(U(\theta))\Big) \Big|_{\theta_1}^{\theta_2} = \int_{\theta_1}^{\theta_2} \cos \eta f^x(U(\eta)) + \sin \eta f^y(U(\eta)) d\eta.
\end{alignat}
The following lemma from \cite{ER} shows there is a version of $U$ such that these conditions will hold everywhere, not just almost everywhere.  From now on, consider this version of $U$.

\begin{lemma} \cite{ER}
    Let $\Omega\subset\R^n$ measurable nonempty, $K\subset\R^m$ compact, $U\in L^\infty(\Omega)$ so that $U(x)\in K$ for a.e.\ $x\in\Omega$,
    $g:K\rightarrow\R^k$ and $\tilde g:\Omega\rightarrow\R^k$ continuous. 
    If
    \begin{alignat}{1}
        g(U(x)) &\leq \tilde g(x) \quad\text{for a.e.\ $x\in\Omega$,} \myeqlabel{eq:gg}
    \end{alignat}
    (meaning $g_i(U(x))\leq \tilde g_i(x)$ for all $i$, where $g=(g_1,...,g_k)$, $\tilde g=(\tilde g_1,...,\tilde g_k)$),
    then we can find a version $\tilde U$ of $U$, with values in $K$ \emph{everywhere}, so that 
    \begin{alignat}{1}
        g(\tilde U) &\leq \tilde g \quad\text{for \emph{all} $x\in\Omega$.}\myeqlabel{eq:ggg}
    \end{alignat}
\end{lemma}

\section{Jump Conditions}
We have, for all $\theta_1$ and $\theta_2$,
\begin{alignat}{1}
\Big(-\cos \theta f^y(U(\theta))+ \sin \theta f^x(U(\theta))\Big) \Big|_{\theta_1}^{\theta_2} = \bO(|\theta_2-\theta_1|). \myeqlabel{fluxjump}
\end{alignat}
On sequences $\left\{ \theta^\pm_n \right\} \rightarrow \theta$ with $U(\theta^\pm) \rightarrow U_\pm$, we have
\begin{alignat}{1}
\sin \theta \Big[ f^x(U) \Big] - \cos \theta \Big[f^y(U)\Big] = 0
\end{alignat}
(where $[g(U)] = g(U_+)-g(U_-)$ for any function $g$ of $U$).
For the case of non-isentropic Euler, this becomes
\begin{alignat}{1}
\sin \theta [\rho u] -\cos \theta [\rho v] &= 0 \myeqlabel{mass flux} \\
\sin \theta [\rho u^2 + p] -\cos \theta [ \rho u v] &= 0 \myeqlabel{x-mom-flux} \\
\sin \theta [\rho u v] -\cos \theta [\rho v^2 + p] &=0 \myeqlabel{y-mom-flux} \\
\sin \theta \left[ u(E+p)\right] -\cos \theta \left[ v(E+p)\right] &=0 \myeqlabel{energy-flux}
\end{alignat}
where $E := \rho \left(\frac{u^2+v^2}{2} +e \right)$ is the total energy per unit volume, with $e$ the internal energy per unit mass. \\
To separate the cases of shocks and contacts, we introduce the normal (angular) and tangential (radial) velocities at $\theta$:
\begin{alignat}{1}
N &:=  u \sin \theta -v \cos \theta  \qquad ; \qquad L := u\cos \theta  + v\sin \theta \\
u &= N \sin \theta + L \cos \theta \qquad ; \qquad v = -N \cos \theta + L \sin \theta.
\end{alignat}
We immediately observe that \eqref{mass flux} is equivalent to
\begin{alignat}{1}
[\rho N] = 0. \myeqlabel{normal-mass-flux}
\end{alignat}
\eqref{x-mom-flux} is equivalent to
\begin{alignat}{1}
\sin \theta [p] = [\cos \theta \rho u v - \sin \theta \rho u^2] = -[\rho u N]. \myeqlabel{pressurejump}
\end{alignat}
Similarly, \eqref{y-mom-flux} yields
\begin{alignat}{1}
\cos \theta [p] = [ \sin \theta \rho u v - \cos \theta \rho v^2] = [\rho v N].
\end{alignat}
Therefore,
\begin{alignat}{1}
0 = \sin \theta [\rho v N] + \cos \theta [\rho u N] = [\rho N L].
\end{alignat}
This means that $\rho_+ N_+ L_+ - \rho_- N_- L_- = 0$, or $\rho_+ N_+(L_+-L_-) = 0$ (from \eqref{normal-mass-flux}).  Therefore, if $\rho_+ N_+ = \rho_- N_- \neq 0$, $[L]$ = 0.  Hence, if there is mass flux through a shock, the tangential velocity is continuous.  However, if $N^+ = N^- = 0$, then the tangential velocity may be discontinuous (called a contact discontinuity).  Also, we have
\begin{alignat}{1}
[p] =-\sin \theta [\rho u N] + \cos \theta [\rho v N] = -[\rho N^2] \implies [\rho N^2 + p] = 0 .\myeqlabel{pressurejumpa}
\end{alignat}
Finally, \eqref{energy-flux} is equivalent to
\begin{alignat}{1}
0=[N(E+p)]=\left[\frac{1}{2}\rho N |\vec{u}|^2 + \rho N e + Np \right]. \myeqlabel{energyjumpa}
\end{alignat}
For the case of a shock, divide by $\rho_+ N_+ = \rho_- N_-$ to obtain
\begin{alignat}{1}
\left[\frac{1}{2}|\vec{u}|^2 + e + p \tau\right] = 0, \myeqlabel{energyjump}
\end{alignat}
where $\tau := \rho^{-1}$ is the specific volume (volume per unit mass).  Denote $\rho_- N_- = \rho_+ N_+ =: \mathcal{M}$.  Then \eqref{pressurejumpa} becomes
\begin{alignat}{1}
[p] = -\mathcal{M}[N]. \myeqlabel{pjump}
\end{alignat}  
Therefore,
\begin{alignat}{1}
\frac{[p]}{[\tau]} &= -\mathcal{M} \frac{[N]}{[\tau]} = -\mathcal{M} \frac{N_+ - N_-}{\frac{1}{\rho_+}-\frac{1}{\rho_-}} \\ 
&= -\mathcal{M} \frac{N_+-N_-}{\frac{\rho_- - \rho_+}{\rho_-\rho_+}} = -\mathcal{M} \frac{\rho_-\mathcal{M} - \rho_+ \mathcal{M}}{\rho_--\rho_+} = -\mathcal{M}^2. \myeqlabel{msquared}
\end{alignat}
Multiplying \eqref{pjump} by $(\tau_- + \tau_+)$ yields
\begin{alignat}{1}
[p](\tau_- + \tau_+) &= -\mathcal{M}[N](\tau_- + \tau_+)\\
&= \mathcal{M}(N_- - N_+)\frac{\rho_- + \rho_+}{\rho_- \rho_+}\\
&= \frac{\mathcal{M}(\rho_- N_- - \rho_+ N_+ - \rho_- N_+ + \rho_+ N_-)}{\rho_- \rho_+} \\
&= \mathcal{M} (\tau_- N_- - \tau_+ N_+) = N_-^2 - N_+^2 = - [|\vec{u}|^2].
\end{alignat}
Substituting this into \eqref{energyjump} yields
\begin{alignat}{1}
[e + p\tau] &= \frac{1}{2}[p](\tau_- + \tau_+),\\
[e] &= -\frac{1}{2}(p_-+p_+)[\tau].
\end{alignat}
Therefore, for a shock wave with fixed $(\tau_+,p_+)$, the states $(\tau_-, p_-)$ that can be connected by a shock are defined by $H(\tau_-,p_-)=0$, where $H(\tau,p)$ is the Hugoniot function defined below:
\begin{alignat}{1}
 H(\tau, p) = e(\tau, p) - e(\tau_+,p_+) + \frac{1}{2}(\tau - \tau_+)(p+p_+).
\end{alignat}

\section{Weak form in terms of tangential and normal velocities}
Analogous steps to those used to simplify the jump conditions to be in terms of radial and angular velocities can be done on the weak form of the equations.  Multiplying a distribution (in this case an $L^{\infty}$ function) by a smooth function ($\sin \theta$ or $\cos \theta$) results in another distribution, and the product rule applies for distributional derivatives of distributions multiplied with smooth functions.\\
\eqref{finalweak} with the Euler fluxes becomes the following system:
\begin{alignat}{1}
\Big( \sin \theta \rho u - \cos \theta \rho v\Big)_\theta &= \cos \theta \rho u + \sin \theta \rho v, \\
\Big( \sin \theta (\rho u^2 + p) - \cos \theta (\rho u v)\Big)_\theta &= \cos \theta (\rho u^2 + p) + \sin \theta (\rho u v), \\
\Big( \sin \theta (\rho u v) - \cos \theta (\rho v^2 + p)\Big)_\theta &= \cos \theta (\rho u v) + \sin \theta (\rho v^2 + p), \\
\Big( \sin \theta \big(u( E + p)\big) - \cos \theta \big(v( E + p)\big)\Big)_\theta &= \cos \theta \big(u( E + p)\big) + \sin \theta \big(v( E+p)\big).
\end{alignat}

Substituting the definitions of $N$ and $L$ we obtain:

\begin{alignat}{1}
(\rho N)_\theta &= \rho L, \\
\big( u \rho N + p \sin \theta \big)_\theta &= u \rho L + p \cos \theta, \mylabel{horizweak} \\
\big( v \rho N - p \cos \theta) \big)_\theta &= v \rho L + p \sin \theta, \mylabel{vertweak} \\
\big(N(E+p)\big)_\theta &=  L(E+p).
\end{alignat}

Note that

\begin{alignat}{1}
\sin \theta \big( u \rho N + p \sin \theta \big)_\theta &= \big( u \sin \theta( \rho N) + p \sin^2 \theta \big)_\theta - \cos \theta \big( u \rho N + p \sin \theta \big),  \\
-\cos \theta \big( v \rho N - p \cos \theta) \big)_\theta &= \big( -v \cos \theta (\rho N) + p \cos^2 \theta \big)_\theta - \sin \theta \big( v \rho N - p \cos \theta) \big).
\end{alignat}

Thus, $\sin \theta \cdot \eqref{horizweak} - \cos \theta \cdot \eqref{vertweak}$ yields
\begin{alignat}{1}
\big(\rho N^2 + p\big)_\theta - \rho N L = \rho N L,
\end{alignat}
or
\begin{alignat}{1}
\big(\rho N^2 + p \big)_\theta = 2 \rho N L.
\end{alignat}

Similarly,

\begin{alignat}{1}
\cos \theta \big( u \rho N + p \sin \theta \big)_\theta &= \big( u \cos \theta( \rho N) + p \sin \theta \cos \theta \big)_\theta + \sin \theta \big( u \rho N + p \sin \theta \big),  \\
\sin \theta \big( v \rho N - p \cos \theta) \big)_\theta &= \big( v \sin \theta (\rho N) - p \sin \theta \cos \theta \big)_\theta - \cos \theta \big( v \rho N - p \cos \theta) \big),
\end{alignat}
and so $\cos \theta \cdot \eqref{horizweak} + \sin \theta \cdot \eqref{vertweak}$ yields

\begin{alignat}{1}
\big(\rho L N)_\theta + \rho N^2 + p = \rho L^2 + p,
\end{alignat}
or

\begin{alignat}{1}
\big(\rho L N)_\theta = \rho L^2 - \rho N^2.
\end{alignat}

  Therefore, the Euler equations are equivalent to
\begin{alignat}{1}
(\rho N)_\theta &= \rho L, \\
(\rho N^2 + p)_\theta &= 2 \rho N L, \\
(\rho L N)_\theta &= \rho L^2 - \rho N^2, \\
(N(E+p))_\theta &=  L(E+p),
\end{alignat}
satisfied in the distributional sense.

\section{Shock sides}
For now, consider self-similar solutions $U = (\rho, \rho u, \rho v, E) : S^1 \rightarrow \mathcal{P} \in \R^4$ that are $L^\infty$,  where the phase space $\mathcal{P}$ consists of states $U$ satisfying:
\begin{alignat}{1}
0 < C^{-1} \leq \rho \leq C < \infty \\
0 < C^{-1} \leq e \leq C < \infty \\
||U||_{L^\infty} \leq C < \infty
\end{alignat}
where $C$ is a positive constant.  Recall $E := \rho\left( \frac{u^2+v^2}{2} + e\right)$, where $e$ is the specific internal energy.  Denote by $p$ the pressure, $T$ the temperature, $\tau$ the specific volume, $\rho = \frac{1}{\tau}$ the density, and $S$ the specific entropy.  From thermodynamics, knowing any two of $p, T, \tau, e, S$ determines the other three. \\
We assume that the equation of state for pressure is given by
$p = g(\tau, S)$
and that it is smooth for the phase space under consideration.  We also assume
\begin{alignat}{1}
g_{\tau} &:= -\rho^2 c^2 < 0, \\
g_{\tau \tau} &> 0, \\
g_{S} &> 0,
\end{alignat}
where $c>0$ is the sound speed.  Furthermore, from thermodynamics we have
\begin{alignat}{1}
\left( \frac{\partial e}{\partial \tau} \right)_{S = \textrm{const}} &= -p \\
\left( \frac{\partial e}{\partial S}\right)_{\tau = \textrm{const}} &= T,
\end{alignat}
or equivalently
\begin{alignat}{1}
\left( \frac{\partial e}{\partial \rho} \right)_{S = \textrm{const}} &= \frac{p}{\rho^2} \\
\left( \frac{\partial e}{\partial S}\right)_{\rho = \textrm{const}} &= T.
\end{alignat}
To emphasize that this definition of sound speed,
\begin{alignat}{1}
c^2 := \left( \frac{\partial p}{\partial \rho}\right)_{S = \textrm{const}},
\end{alignat}
agrees with another definition seen in the literature, 
\begin{alignat}{1}
c^2 := \left( \frac{\partial p}{\partial \rho} \right)_{e = \textrm{const}} + \frac{p}{\rho^2} \left( \frac{\partial p}{\partial e} \right)_{\rho = \textrm{const}}, 
\end{alignat}
consider
\begin{alignat}{1}
p(\rho, S) = p(\rho, e).
\end{alignat}
Differentiating with respect to $\rho$ while holding $S$ constant, we obtain
\begin{alignat}{1}
\left( \frac{\partial p}{\partial \rho}\right)_{S = \textrm{const}} &= \left( \frac{\partial p}{\partial \rho}\right)_{e = \textrm{const}} \left(\frac{\partial \rho}{\partial \rho}\right)_{S = \textrm{const}} + \left(\frac{\partial p}{\partial e}\right)_{\rho = \textrm{const}} \left(\frac{\partial e}{\partial \rho}\right)_{S = \textrm{const}}  \\ 
&=  \left( \frac{\partial p}{\partial \rho} \right)_{e = \textrm{const}} + \frac{p}{\rho^2} \left( \frac{\partial p}{\partial e} \right)_{\rho = \textrm{const}}.
\end{alignat}
We now argue that even in the $L^{\infty}$ setting, in which left and right limits may not exist, there still exists a well defined notion of shocks and contacts, as well as front and back sides of shocks.  For the remainder of this paper, a subscript $max$ or $min$ on a quantity refers to the maximum or minimum permissible value of that quantity for $U \in \mathcal{P}$.\\
\begin{lemma} 
\mylabel{forwback}
 If $U$ is discontinuous at $\theta_0$, then it can be well defined as having either a forward facing shock, a backward facing shock, or a contact discontinuity at $\theta_0$.
\end{lemma}
\begin{proof}
Suppose $U(\theta)$ is discontinuous at $\theta_0$.  Then we can pick sequences $\left\{ \theta_n \right\}, \left\{ \theta'_n \right\} \rightarrow \theta_0$ with $U(\theta_n) \rightarrow U_0$, $U(\theta'_n) \rightarrow U'_0$, with $U_0 \neq U'_0$.  Then, from above, necessarily $\rho_0 N_0 = \rho'_0 N'_0$.  Based on our assumptions about $\rho$, it follows that $N_0$ and $N'_0$ are either both positive, both negative, or both zero.  \\
\vspace{2pc}
\\
Case 1: $N_0 = N'_0 = 0.$  Let $\left\{\theta''_n \right\}$ be any other sequence converging to $\theta_0$, and take any subsequence $\theta''_{n(k)}$.  Since $U$ is $L^{\infty}$, there exists a subsequence $\left\{ \theta''_{n(k(j))} \right\}$ such that $U(\theta''_{n(k(j))}) \rightarrow U''_0$.  Applying the jump conditions to $U_0$ and $U''_0$ shows that $N''_0 = 0$, because $N_0 = 0$.  Therefore, since any subsequence has a subsequence converging to zero, $N(\theta''_n) \rightarrow 0$, and as $\left\{ \theta''_n \right\}$ was arbitrary, we see that in fact $N(\theta)$ is continuous at $\theta_0$ and $N(\theta_0) = 0$.  In this case we say that $U$ has a contact discontinuity at $\theta_0$. \\
\vspace{2pc}
\\
Case 2: $N_0, N_0' >0$.  We claim that there exists some neighborhood containing $\theta_0$ on which $N(\theta)$ is positive.  Suppose not.  Then there exists $\left\{ \theta''_n \right\} \rightarrow \theta_0$ such that $N(\theta''_n) \leq 0$ for all $n$.  Again, there exists a subsequence $\left\{\theta''_{n(k)} \right\}$ such that $U(\theta''_{n(k)}) \rightarrow U''_0$, and by assumption $N''_0 \leq 0$.  However, applying the jump conditions to $N_0$ and $N''_0$ gives a contradiction.  Therefore, $N(\theta) > 0$ on some neighborhood of $\theta_0$, and we call the shock forward facing.  (The interpretation is that gas particles enter the front side of the shock, and leave the back side.  Since the flow of mass through the shock is aligned with our choice of normal vector for the shock, we call this case forward facing.)\\
\vspace{2pc}
\\
Case 3, for which $N_0, N_1 < 0$, is similar.  We call this case backward facing.
\end{proof}

\section{Entropy}
The entropy admissibility criterion, with $\theta_n^- < \theta_n^+$, takes the form
\begin{alignat}{1}
\Big(- \cos \theta \psi^y(U(\theta))+ \sin \theta \psi^x(U(\theta))\Big) \Big|_{\theta_n^-}^{\theta_n^+} + \bO(|\theta_n^+-\theta_n^-|) \leq 0.
\end{alignat}
(Note this is reversed from the usual conservation laws literature since we are using a concave (physical), not convex (mathematical) entropy.)  In the limit $\theta_n^{\pm} \rightarrow \theta$, if $U(\theta_n^\pm) \rightarrow U_\pm$, this becomes
\begin{alignat}{1}
\sin \theta [ \psi^x(U)] - \cos \theta [\psi^y(U)] \leq 0
\end{alignat}
For the case of the Euler equations, this reads
\begin{alignat}{1}
0 \geq \sin \theta [\rho u S] - \cos \theta [\rho v S] =  [\rho N S],
\end{alignat}
which becomes
\begin{alignat}{1}
 \mathcal{M}[S] \leq 0.
\end{alignat}
(where $\mathcal{M} = \rho_- N_- = \rho_+ N_+$).  In the case of a forward facing shock, $S_+ \leq S_-$, and in the case of a backward facing shock, $S_+ \geq S_-$.  In either case, the entropy cannot decrease when the gas passes through the shock from the front to the back.

We now assume the case of a forward facing shock, and fix a pair of left and right sequences $\theta_n^\pm$ as usual, so that $\theta_n^- < \theta_n^+$ for all $n$, and $U(\theta_n^\pm) \rightarrow U_\pm$.  We know that $U_-$ must satisfy $H(\tau_-, p_-) = 0$ for the Hugoniot function for state $(\tau_+, p_+)$.

\begin{lemma}(\cite{courant-friedrichs})\mylabel{supersubsonic}
Assume that for fixed $(\tau_+, p_+)$, the set of states $(\tau,p)$ for which $H(\tau,p)=0$ is a smooth curve in the $(\tau,p)$ plane that can be described by $p=G(\tau)$.  Then, for the case of an entropy admissible forward facing shock, which requires $S_- > S_+$, we have
\begin{enumerate}
\item $\tau_- < \tau_+$ (that is, shocks are compressive),
\item $N_+ > c_+ >0, 0<N_- < c_-.$ 
\end{enumerate}
We have the analogous statement for an entropy admissible backward facing shock, which instead requires $S_- < S_+$:
\begin{enumerate}
\item $\tau_- > \tau_+$,
\item $-N_- > c_- > 0, 0<-N_+ < c_+.$ 
\end{enumerate}
Moreover, for any shock,
\begin{alignat}{1}
|N_\pm| \geq \frac{c_{\textrm{min}}\rho_{\textrm{min}}}{\rho_{\textrm{max}}}.
\end{alignat}
\end{lemma}
\begin{proof}
This proof can be found in \cite{courant-friedrichs}.  First we derive some properties from the equation of state $p = g(\tau, S)$.  Recall
\begin{alignat}{1}
g_{\tau} = -c^2 \rho^2 &< 0, \myeqlabel{gtau} \\
g_{\tau \tau} &> 0, \myeqlabel{gtautau}\\
g_S &> 0. \myeqlabel{gS}
\end{alignat}
Differentiating $S=S(\tau,g(\tau,S))$ with respect to $S$ we obtain
\begin{alignat}{1}
1 = S_p g_S \implies S_p > 0, \myeqlabel{Sp}
\end{alignat}
from \eqref{gS}.  Differentiating $S=S(\tau,g(\tau,S))$ with respect to $\tau$ yields
\begin{alignat}{1}
0 = S_\tau + S_p g_\tau \implies S_\tau > 0, g_\tau = -\frac{S_\tau}{S_p},
\end{alignat}
from \eqref{Sp} and \eqref{gtau}.  Differentiating again with respect to $\tau$ yields
\begin{alignat}{1}
0 = S_{\tau \tau} + 2S_{\tau p}g_\tau + S_p g_{\tau \tau} + S_{pp}g_\tau^2.
\end{alignat}
From \eqref{gtautau}, it follows that
\begin{alignat}{1}
S_{\tau \tau} + 2 S_{\tau p}g_\tau + S_{pp} g_{\tau} < 0.
\end{alignat}
Substituting in $\displaystyle g_\tau = -\frac{S_\tau}{S_p}$ and multiplying by $S_p^2$ yields
\begin{alignat}{1}
S_{\tau \tau}S_p^2 - 2S_{\tau p}S_\tau S_p + S_{pp}S_\tau^2 < 0.
\end{alignat}
We now claim that $G''(\tau_+) > 0$.  We start differentiating the Hugoniot function with respect to $\tau$ along the curve $p=G(\tau)$.  
\begin{alignat}{1}
e(\tau,S(\tau)) - e(\tau_+,S(\tau_+)) + \frac{1}{2}(\tau - \tau_+)(G(\tau)+G(\tau_+)) &= 0 \\
\implies e_S S'(\tau)+e_{\tau} + \frac{1}{2} (G(\tau)+G(\tau_+)) + \frac{1}{2} (\tau-\tau_+)G'(\tau) &=0. \myeqlabel{Htau}
\end{alignat}
Substituting in $e_S = T, e_\tau = -p$ yields
\begin{alignat}{1}
T S'(\tau) + \frac{1}{2} (G(\tau_+)-G(\tau)) + \frac{1}{2}(\tau - \tau_+)G'(\tau) = 0.
\end{alignat}
Therefore, $S'(\tau_+) = 0$.  Furthermore,
\begin{alignat}{1}
T'(\tau)S'(\tau)+TS''(\tau)-\frac{1}{2}G'(\tau)+\frac{1}{2}G'(\tau) + \frac{1}{2} (\tau-\tau_+)G''(\tau) = 0, \myeqlabel{Stautau}
\end{alignat}
implies $S''(\tau_+) = 0.$  The relation $G(\tau) = g(\tau,S(\tau))$ yields
\begin{alignat}{1}
G'(\tau) &= g_\tau + g_S S'(\tau), \\
G''(\tau) &= g_{\tau \tau} + 2 g_{\tau S} S'(\tau) + g_{SS}S'(\tau)^2 + g_S S''(\tau),
\end{alignat}
and so $G''(\tau_+) = g_{\tau \tau}(\tau_+, S(\tau_+)) > 0$ from our assumptions about the equation of state.\\
Now, consider any other state $(\tau_-,p_-), p_- = G(\tau_-)$, such that $H(\tau_-,p_-) = 0$.  Parametrize the straight line segment between these two states as
\begin{alignat}{1}
\tau = \tau_+ + (\tau_- - \tau_+)s,  \qquad p = p_+ + (p_- - p_+)s.
\end{alignat}
($H(\tau)$ and $S(\tau)$ refer to states on the Hugoniot curve, while $H(s)$ and $S(s)$ refer to states on this straight line segment.)  We have that
\begin{alignat}{1}
H'(s)=TS'(s) + \frac{1}{2}\big((p_+-p_-)(\tau_--\tau_+)-(\tau_--\tau_+)(p_--p_+)\big) = TS'(s).
\end{alignat}
Hence if $H'(s)=0$, $S'(s) = 0$.  $H(0)=H(1)=0$, and so $H'(s_0) = 0$ for at least one $s_0$ between 0 and 1.  Therefore, $S'(s_0) = 0$ for at least one $s_0$ between 0 and 1.  However, consider $S''(s)$.
\begin{alignat}{1}
S'(s) &= S_\tau \tau'(s) + S_p p'(s) => \frac{S_\tau}{S_p} \Big|_{s=s_0} = -\frac{p_--p_+}{\tau_--\tau_+}. \\
S''(s) &= S_{\tau \tau} \tau'(s)^2 + 2S_{\tau p} \tau'(s) p'(s) + S_{pp} p'(s)^2 \\
&= (p_--p_+)^2 \left( S_{\tau \tau} \frac{(\tau_--\tau_+)^2}{(p_--p_+)^2} + 2 S_{\tau p} \frac{\tau_--\tau_+}{p_--p_+} + S_{pp} \right), \\
S''(s_0) &= \frac{(p_--p_+)^2}{S_\tau^2} \left(S_{\tau \tau}S_p^2 - 2S_{\tau p}S_\tau S_p + S_{pp}S_\tau^2 \right) < 0.
\end{alignat}
Therefore, $S(s)$ has a single maximum at $s_0$ and no other critical points.  Therefore, $S'(0)>0$, $S'(1)< 0$.  Therefore, $H'(1)<0$ and this straight line cannot be tangent to the Hugoniot curve at $s=1$, or equivalently $\tau = \tau_-$. \\
Now, consider the Hugoniot curve.  We have that
\begin{alignat}{1}
H'(\tau_-)= TS'(\tau_-)+\frac{1}{2} \big( (p_+-p_-) + (\tau_- - \tau_+)G'(\tau_-)\big) = 0.
\end{alignat}
Suppose $S'(\tau_-)$ = 0, then 
\begin{alignat}{1}
G'(\tau_-) = \frac{p_--p_+}{\tau_- - \tau_+}.
\end{alignat}
That is, the Hugoniot curve is tangent to the straight line between $(\tau_+,p_+)$ and $(\tau_-,p_-)$.  This contradicts what we showed above, and so 
\begin{alignat}{1}
S'(\tau) \neq 0 \mbox{ for all } \tau \neq \tau_+.
\end{alignat}
Differentiating \eqref{Stautau} we obtain
\begin{alignat}{1}
T''(\tau)S'(\tau)+2T'(\tau)S''(\tau)+T S'''(\tau) +\frac{1}{2}((\tau-\tau_+)G'''(\tau)+G''(\tau)) = 0,
\end{alignat}
and at $\tau_+$ this yields
\begin{alignat}{1}
S'''(\tau_+)<0.
\end{alignat}
Thus for $(\tau_- - \tau_+)$ sufficiently small,
\begin{alignat}{1}
S_- &> S_+ \iff \tau_- < \tau_+\\
S_- &< S_+ \iff \tau_- > \tau_+
\end{alignat}
which allows us to conclude $S'(\tau_+) = 0$, $S'(\tau) < 0$ for all $\tau \neq \tau_+$.  In conclusion, for a forward facing shock so that $\tau_+$ is a limit from the front side, $S_- \geq S_+$ if and only if $\tau_- \leq \tau_+$, and we can further conclude that $S_- > S_+$ if there is an actual discontinuity.\\
Finally, since $S_s = S_p(p_--p_+) + S_\tau (\tau_--\tau_+)$, the conditions $S(0)>0, S(1)<0$ when combined with $\displaystyle g_\tau = -\frac{S_\tau}{S_p} = -\rho^2 c^2$ and $S_p > 0$ yield
\begin{alignat}{1}
(p_--p_+) - (\rho_+c_+)^2 (\tau_--\tau_+) &> 0, \\
(p_--p_+) - (\rho_-c_-)^2 (\tau_--\tau_+) &< 0,
\end{alignat}
leading to (taking $\tau_- < \tau_+$)
\begin{alignat}{1}
(\rho_+ c_+)^2 < - \frac{p_+-p_-}{\tau_+-\tau_-} < (\rho_- c_-)^2.
\end{alignat}
\eqref{msquared} yields
\begin{alignat}{1}
(\rho_+ c_+)^2 < m^2 < (\rho_- c_-)^2,
\end{alignat}
from which we obtain Lax-type conditions $N_+ > c_+>0$, $0<N_- < c_-$. \\
The case of a backward-facing shock is similar, and the remaining lower bound on $|N_{\pm}|$ follows from the fact that one must be greater in magnitude than $c_{\textrm{min}}$, and the relation
\begin{alignat}{1}
N_{\pm} = N_{\mp}\frac{\rho_\mp}{\rho_\pm}.
\end{alignat}
\end{proof}

\section{Properties of discontinuities, uniform distances between different types of discontinuities}
We now use the idea of averaged matrices, that is smooth matrix valued functions of $U^\pm$ and $\theta$ that satisfy:
\begin{enumerate}[(a)]
\item $\sin\theta \big(f^x(U^+)-f^x(U^-)\big) - \cos \theta \big(f^y(U^+)-f^y(U^-)\big) $\\
 $ = \hat{A}(U^-,U^+;\theta)(U^+-U^-)$

\item $\hat{A}(U,U;\theta) = A(U;\theta):=  \sin \theta f^x_U(U) - \cos \theta f^y_U(U),$

\item $\hat{A}(U^-,U^+;\theta)$ is diagonalizable with real eigenvalues for all $U^\pm \in \mathcal{P}$, $\theta \in [0, 2\pi[$.

\end{enumerate}

The existence of matrices with the important conservation and diagonalizability properties is guaranteed (see \cite{HL}) if our system possesses an entropy-entropy flux $(\eta,\psi^x,\psi^y)$ with $\eta$ positive definite, but one that is easier for computations in the polytropic setting uses a different line integral in phase space as opposed to one using the entropy gradient as described in \cite{HL}.  It has the property that the averaged matrix is simply the flux matrix evaluated at a suitable averaged state.  This is called the Roe linearization for the system and is very common in numerics.  \\  
Rearranging \eqref{fluxjump}, we obtain
\begin{alignat}{1}
\sin\theta_1 \Big(f^x\big(U(\theta_2)\big)-f^x\big(U(\theta_1)\big)\Big) - \cos \theta_1 \Big(f^y\big(U(\theta_2)\big)-f^y\big(U(\theta_1)\big)\Big) = \bO(|\theta_2-\theta_1|).
\end{alignat}

Using the averaged matrix, we have
\begin{alignat}{1}
\hat{A}\big(U(\theta_2),U(\theta_1);\theta_1 \big) \big(U(\theta_2)-U(\theta_1)\big) = \bO(|\theta_2-\theta_1|). \myeqlabel{Adiff}
\end{alignat}

Now, consider the matrix
\begin{alignat}{1}
A(U;\theta) = \sin \theta f^x_U(U)- \cos \theta f^y_U(U).
\end{alignat}
By direct computation, we have that
\begin{alignat}{1}
\mbox{det}\Big(\sin \theta f^x_U(U)- \cos \theta f^y_U(U)\Big) = 0 \iff |N|=c \mbox{ or } N = 0.
\end{alignat}
We then have the following theorem.
\begin{theorem} \mylabel{newthm1}
Suppose $U$ is continuous on an interval $]\theta_1,\theta_2[$ and that $|N| \neq \pm c$ or $0$ on this interval.  Then $U$ is constant on this interval.
\end{theorem}
\begin{proof}
Fix some $\theta \in ]\theta_1,\theta_2[$.
    We claim that $U$ must be 
    Lipschitz at $\theta$.  
    Suppose not.  
    Then we can choose a sequence $\left\{\theta_n \right\} \rightarrow \theta$ (with $|\theta_n-\theta| \neq 0$) such that 
    $$ 0 < \left|\frac{U(\theta_n)-U(\theta)}{\theta_n-\theta}\right| \nearrow\infty. $$
    Divide both sides of \eqref{Adiff} by $|U(\theta_n)-U(\theta)|$ to obtain
    \begin{alignat}{1}&
\hat{A}\big(U(\theta),U(\theta_n);\theta \big)\left(\frac{U(\theta_n)-U(\theta)}{|U(\theta_n)-U(\theta)|}\right)        \notag\\&= 
\frac{1}{|U(\theta_n)-U(\theta)|} \bO(|\theta_n-\theta|)        = o(1)\quad\text{as $n\rightarrow\infty$.}
    \end{alignat}
    By assumption, $\sin\theta f^x_U(\theta) - \cos\theta f^y_U(\theta)$ is regular, 
    so for $|\theta_n-\theta|$ sufficiently small, $\hat A\big(U(\theta),U(\theta_n);\theta\big)$ will be uniformly regular, because the eigenvalues of $\hat{A}$ are continuous functions of $U^\pm$  (see for example \cite{serrematrix}), and $U$ is continuous at $\theta$ by assumption.  That is,
    \[ \exists \delta>0~\forall z\in\R^m:\Big|\Big(\hat A\big(U(\theta),U(\theta_n);\theta_1\big)\Big) z\Big| \geq \delta|z| \]
    Taking $n \rightarrow \infty$, the left hand side stays bounded away from zero, 
    while the right hand side goes to zero, 
    leading to a contradiction.  

    Therefore, $U$ must be Lipschitz on $]\theta_1,\theta_2[$.
    Assuming $\theta$ is a point of differentiability of $U$, we obtain
    $$ \Big(\sin \theta f^x_U\big(U(\theta)\big)-\cos \theta f^y_U\big(U(\theta)\big)\Big) U_\theta = 0.$$
    However, as we assumed the matrix was regular on $]\theta_1,\theta_2[$, it follows that $U_\theta=0$ on this interval.
    A Lipschitz function is the integral of its derivative, so $U$ is constant on $]\theta_1,\theta_2[$.
\end{proof}
\begin{theorem} \mylabel{newthm2}
Suppose there is a shock at $\theta_0$.  Then there exist $\sigma^+(\theta_0) > \theta_0$ and $\sigma^-(\theta_0) < \theta_0$ so that $U$ is constant on $]\sigma^-(\theta_0),\theta_0[, ]\theta_0,\sigma^+(\theta_0)[$.  In particular U has well defined left and right limits at shocks.
\end{theorem}
\begin{proof}
We consider the case of a forward facing shock; backward facing shocks can be treated similarly.  If there is a shock, choose left and right sequences $\theta_n^\pm \rightarrow \theta_0$ with $U(\theta^\pm_n) \rightarrow U_\pm$ with $\theta_n^- < \theta_n^+$ for all n.  We have from above that
\begin{alignat}{1}
N_+ &> c_+ \\
c_- &> N_- > 0.
\end{alignat}
Suppose there is no $\sigma^+(\theta_0)$ such that $N > c$ for all $\theta \in ]\theta_0, \sigma^+(\theta_0)[$.  Then pick a new sequence $\theta_n^{++} \searrow \theta_0$ such that (passing to subseqences if necessary)
\begin{alignat}{1}
N(\theta_n^{++}) &\leq c(\theta_n^{++}) \\
U(\theta_n^{++}) &\rightarrow U_{++} \\
\theta_n^+ &< \theta_n^{++}
\end{alignat}
for all $n$.  If $U_{++} \neq U_+$, then the Lax condition requires $N_{++} > c_{++}$, which contradicts our construction.  If $U_{++} = U_+$, then $N_{++} = N_{+} > c_{+} = c_{++}$, also a contradiction.  Therefore there exists $\sigma^+(\theta_0) > \theta_0$ such that $N > c$ for all $\theta \in ]\theta_0, \sigma^+(\theta_0)[$.  \\

The fact that there is some $\sigma'^-(\theta_0) < \theta_0$ such that $N< c$ for all $\theta \in ]\sigma'^-(\theta_0),\theta_0[$ is analogous.  \\

As observed in the proof of Lemma \myref{forwback}, $N>0$ for some neighborhood of $\theta_0$, otherwise the jump conditions cannot be satisfied.   Taking $\sigma^-(\theta_0)$ to be the left endpoint of this neighborhood intersected with $]\sigma'^-(\theta_0),\theta_0[$ shows that $0 < N < c$ for all $\theta \in ]\sigma^-{\theta_0},\theta_0[$. \\
All that is left to show is that $U$ cannot possess any other discontinuities in these left and right neighborhoods.  Contacts are already ruled out - we know that $N = 0$ at a contact.  Therefore we must show there cannot be shocks.  Suppose there was a shock at $\theta_1 \in ]\theta_0, \sigma^+(\theta_0)[$.  It too must be forward facing, since $N$ is positive for all such $\theta$.  Then, for some $\eta \in ]\theta_0,\sigma^+(\theta_0)[ \cap ]\sigma^-(\theta_1),\theta_1[$ we would have
\begin{alignat}{1}
N(\eta) > c(\eta) < N(\eta),
\end{alignat}
a contradiction.  Therefore, $U$ is continuous on $]\theta_0, \sigma^+(\theta_0)[$ and Theorem \myref{newthm1} shows it is constant.  Similarly we can conclude $U$ is constant on $]\sigma^-(\theta_0),\theta_0[$. 
\end{proof}
Similarly, contact discontinuities are isolated and must also possess constant neighborhoods.  
\begin{theorem} \mylabel{newthm3}
Suppose $U$ has a contact discontinuity at $\theta_0$.  Then, either (a) there exist $\sigma^+(\theta_0) > \theta_0$ and $\sigma^-(\theta_0) < \theta_0$ such that $U$ is constant on $]\sigma^-(\theta_0),\theta_0[, ]\theta_0,\sigma^+(\theta_0)[$, or (b) $\theta_0$ is contained in a closed interval on which $|\vec{u}| = 0$. 
\end{theorem}
\begin{proof}
From previous analysis we know that in fact $N$ is continuous at $\theta_0$ and that $N(\theta_0) = 0$.  Therefore, choose a $ \pi > \delta > 0$ such that
\begin{alignat}{1}
|\theta - \theta_0| < \delta \implies |N(\theta)| \leq \frac{c_{\textrm{min}} \rho_{\textrm{min}}}{2 \rho_{\textrm{max}}}.
\end{alignat}
By the last part of Lemma \myref{supersubsonic}, we have the lower bound for normal velocity at a shock:
\begin{alignat}{1}
|N_\pm| \geq \frac{c_\textrm{min} \rho_\textrm{min}}{\rho_\textrm{max}}.
\end{alignat}
  Therefore, for $|\theta-\theta_0| < \delta$ there can be no shocks.  If there is another contact discontinuity, $N$ must still be continuous, and so $N$ is continuous for $|\theta - \theta_0| < \delta$.  Therefore, the set
  \begin{alignat}{1}
  C := \left\{ \theta \Big| |\theta - \theta_0| < \delta, N(\theta) = 0 \right\}
  \end{alignat}
  is a closed set in $]\theta_0 - \delta, \theta_0 + \delta[$.  Its complement in $]\theta_0 - \delta, \theta_0 + \delta[$ is therefore open and is thus a countable union of open intervals.  If one of these open intervals $]a_i, b_i[$ has neither endpoint not equal to $\theta_0 \pm \delta$, then $U$ is constant on $(a_i,b_i)$, and so $U \equiv U_i$ on $]a_i, b_i[$, and $N(a_i) = N(b_i)=0$.  However, since $N$ is continuous we can take limits and find
  \begin{alignat}{1}
  N(a_i) &= u_i \sin a_i - v_i \cos a_i  = 0\\
  N(b_i) &= u_i \sin b_i - v_i \cos b_i = 0
  \end{alignat}
  Then, since $a_i \neq b_i$, the vectors $(\sin a_i, -\cos a_i)$ and $(\sin b_i, -\cos b_i)$ span $\R^2$, and so $(u_i, v_i) = 0$.  However this contradicts that $N \neq 0$ on $]a_i, b_i[$.
  Therefore, $C$ is either a closed interval containing $\theta_0$ or simply $\left\{\theta_0\right\}$ .  Consider the conservation of mass equation,
  \begin{alignat}{1}
  (\rho N)_\theta = \rho L,
  \end{alignat}
  in the distributional sense.
  If $N \equiv 0$ on a closed interval, we can take a strong derivative at any $\theta$ in its interior to obtain that $L \equiv 0$ on the interior.  Therefore the supposed contact at $\theta_0$ could not have had a jump in tangential velocity, only density and entropy (continuity of pressure is required by the jump conditions).  (This opens the possibility to very irregular solutions.  If the velocity field is zero on some interval, then a completely arbitrary density distribution can be prescribed, as long as the entropy/internal energy/temperature is also prescribed to result in constant pressure).  If we disregard this possibility, then it follows that this closed interval must be a single point, and so there can only be one contact for $|\theta - \theta_0| < \delta$.  These very irregular solutions are not all that surprising, considering that when the velocity is identically zero the Euler equations become $p_\theta = 0$ in the sense of distributions, so $p \equiv$ constant.  Note that for isentropic flow, this situation cannot occur, since constant pressure can only be attained if density is constant.
\end{proof}

Note we have shown that the set on which $U$ has a discontinuity is countable and discrete.  Therefore, right and left limits are well defined, and so from here on out we modify $U$ to be right continuous at every point.

\begin{lemma} \mylabel{finitecontact}
Assume there are no stagnation points (that is, $|\vec{u}| \neq 0$ everywhere).  Then the set on which $N=0$ is a finite set of points.  Moreover, if $N(\theta_0) = 0$, and there is a shock at $\theta'$,
\begin{alignat}{1}
|\theta'-\theta_0| \geq \delta > 0,
\end{alignat}
with $\delta$ independent of $U$.
\end{lemma}
\begin{proof}
Take any $\theta_0$ with $N(\theta_0) = 0.$  Define
\begin{alignat}{1}
\sigma^+(\theta_0) := \sup_{\eta>\theta_0} \left\{ \eta \, \Big| \, 0 < |N(\theta)| < \frac{c_{\textrm{min}}\rho_{\textrm{min}}}{2\rho_{\textrm{max}}} \quad \forall \theta \in ]\theta_0, \eta[ \right\}.
\end{alignat}
This supremum is defined because it is taken over a nonempty set by Theorem \myref{newthm3}.  Moreover, since there cannot be any shocks on $]\theta_0, \sigma^+(\theta_0)]$, it follows that $|N(\sigma^+(\theta_0))| = 0$ or $\frac{c_{\textrm{min}}\rho_{\textrm{min}}}{2\rho_{\textrm{max}}}.$  We claim that if $\sigma^+(\theta_0) - \theta_0 < \pi$, then $|N(\sigma^+(\theta_0))| = \frac{c_{\textrm{min}}\rho_{\textrm{min}}}{2\rho_{\textrm{max}}}$.  If not, then since there could be no shocks or contacts between $\theta_0$ and $\sigma^+(\theta_0)$, by Theorem \myref{newthm1}, $U$ would be constant.  However, this would lead to a contradiction of $N(\theta_0) = N(\sigma^+(\theta_0)) = 0$ and they were separated by less than $\pi$.  Therefore, $|N(\sigma^+(\theta_0))|= \frac{c_{\textrm{min}}\rho_{\textrm{min}}}{2\rho_{\textrm{max}}}$, and $U$ must be constant on $]\theta_0, \sigma^+(\theta_0)]$ (and continuous at $\sigma^+(\theta_0)$).  When $U$ is constant, $\frac{dN}{d\theta} = L$, and so
\begin{alignat}{1}
\frac{c_{\textrm{min}}\rho_{\textrm{min}}}{2\rho_{\textrm{max}}} &= |N(\sigma^+(\theta_0))-N(\theta_0)| \\
&= \lim_{\eta \searrow \theta_0} \Big| \int_\eta^{\sigma^+(\theta_0)} L(\phi) d\phi \Big| \leq |\vec{u}|_{\textrm{max}}(\sigma^+(\theta_0)-\theta_0).
\end{alignat}
Therefore, $(\sigma^+(\theta_0)-\theta_0)$ is bounded below independent of $U$, and another contact could only happen for $\theta > \sigma^+(\theta_0)$.  A similar argument works for $\theta < \theta_0$, and so the total number of contacts is finite. \\
This same calculation shows that the distance between a contact and a shock is lower bounded independent of $U$.
\end{proof}

\begin{lemma} \mylabel{unifseparation}
If there is a forward facing shock at $\theta$, and a backward facing shock at $\theta'$, then
\begin{alignat}{1}
|\theta - \theta'| \geq \delta > 0
\end{alignat}
for some $\delta$ independent of $U$.
\end{lemma}
\begin{proof}
For forward facing shocks, the normal velocity is positive on either side.  Similarly, it is negative on either side of a backward facing shock.  Therefore, between a forward facing and backward facing shock, $N$ must transition through zero, not jump between positive and negative values.  In the previous lemma, we showed that the distance between a point at which $N=0$ and any kind of shock is uniformly bounded away from zero (independent of $U$), and so the claim follows.
\end{proof}

\section{Shock strengths and neighborhood sizes}
We now turn to the case of a polytropic gas, for which
\begin{alignat}{1}
e = \frac{p}{(\gamma-1)\rho}.
\end{alignat}
We assume that $\gamma > 1$, recalling for air $\gamma = 1.4$.  
We now find expressions for various quantities at a shock, following \cite{tran}.  Substituting the expression for $e$ into \eqref{energyjumpa} we obtain
\begin{alignat}{1}
\Big[ \frac{1}{2}\rho N (N^2 + L^2) + \frac{\gamma}{\gamma - 1} N p \Big]=0.
\end{alignat}
However, $[\rho N L^2]=0$ for a shock, and so we have
\begin{alignat}{1}
\Big[ \frac{1}{2} \rho N^3 + \frac{\gamma}{\gamma -1} N p \Big]=0.
\end{alignat}
This can be rewritten as (recalling $\rho_- N_- = \rho_+ N_+$)
\begin{alignat}{1}
0 = \frac{1}{2} (\rho_+ N_+) N_+^2 + \frac{\gamma}{\gamma-1} N_+ p_+ - \frac{1}{2} (\rho_+  N_+) N_-^2 - \frac{\gamma}{\gamma-1} N_- p_-.
\end{alignat}
However, since
\begin{alignat}{1}
N_- = N_+ + \frac{p_+-p_-}{\rho_+ N_+},
\end{alignat}
we obtain
\begin{alignat}{1}
0 &= \frac{1}{2} (\rho_+ N_+) N_+^2 + \frac{\gamma}{\gamma-1} N_+ p_+ \\
&- \frac{1}{2} (\rho_+  N_+) \left(N_+ + \frac{p_+-p_-}{\rho_+ N_+}\right)^2 - \frac{\gamma}{\gamma-1}  p_- \left(N_+ + \frac{p_+-p_-}{\rho_+ N_+}\right).
\end{alignat}
Then,
\begin{alignat}{1}
0 &= \frac{\gamma}{\gamma-1} N_+ (p_+-p_-) - N_+(p_+-p_-) - \frac{1}{2} \frac{(p_+-p_-)^2}{\rho_+ N_+} - \frac{\gamma}{\gamma-1}p_- \frac{p_+-p_-}{\rho_+ N_+} \\
&= \frac{p_+-p_-}{\rho_+ N_+} \left( \frac{1}{\gamma-1} \rho_+ N_+^2 - \frac{1}{2} (p_+-p_-) - \frac{\gamma}{\gamma-1}p_- \right) \\
&= \frac{p_+-p_-}{(\gamma-1)\rho_+ N_+} \left( \rho_+ N_+^2 - \frac{\gamma-1}{2}(p_+-p_-)-\gamma p_-\right) \\
&= \frac{p_+-p_-}{(\gamma-1)\rho_+ N_+} \left( \rho_+ N_+^2 - \frac{\gamma-1}{2} p_+ - \frac{\gamma+1}{2}p_- \right).
\end{alignat}
Therefore,
\begin{alignat}{1}
N_+^2 &= \frac{1}{2 \rho_+} \Big(p_+(\gamma - 1) + p_-(\gamma + 1)\Big) \\
&= \frac{p_+}{\rho_+} \left( \frac{\gamma-1}{2} + \frac{p_-}{p_+} \frac{(\gamma + 1)}{2} \right).
\end{alignat}
Introduce
\begin{alignat}{1}
z = \frac{p_--p_+}{p_+} > 0
\end{alignat}
as the shock strength of a forward facing shock, so that
\begin{alignat}{1}
p_- = p_+(1+z).
\end{alignat}
The remainder of these calculations are for forward facing shocks, and the same statements hold for backward facing shocks but with $\pm$ switched, and all normal velocities are negative.
\begin{alignat}{1}
N_+^2 &=  \frac{p_+}{\rho_+} \left( \frac{\gamma-1}{2} +  \frac{\gamma + 1}{2} (1+z)\right) \\
&= \frac{\gamma p_+}{\rho_+}\left(1 + z\frac{\gamma+1}{2\gamma}\right)\\
&= c^2_+ \left(1 + z\frac{\gamma+1}{2\gamma}\right).
\end{alignat}
Therefore we have the relations (assuming a forward facing shock)
\begin{alignat}{1}
N_+ &= c_+ \sqrt{1+z\frac{\gamma+1}{2\gamma}} \\
N_+ - c_+ &= c_+ \left( \sqrt{1+z\frac{\gamma+1}{2\gamma}} -1 \right).
\end{alignat}
We need similar relations for the state behind the shock.
\begin{alignat}{1}
N_- &= \frac{1}{\rho_+ N_+} \left(\rho_+ N_+^2 - p_- + p_+\right) \\ 
&= \frac{1}{\rho_+ N_+} \left( \rho_+ c_+^2 \left( 1 + z\frac{\gamma+1}{2 \gamma}\right) - p_+(1+z)+p_+\right) \\
&= \frac{c_+^2}{N_+} \left(1 + z \frac{\gamma+1}{2\gamma} - \frac{z p_+}{\rho_+c_+^2} \right) \\
&= \frac{c_+^2}{N_+} \left(1 + z \frac{\gamma+1}{2\gamma} - \frac{z}{\gamma} \right) \\
&= \frac{c_+^2}{N_+} \left(1+z\frac{\gamma-1}{2\gamma}\right).
\end{alignat}
Next, we obtain expressions for $\rho_-$, then $c_-$, $N_- - c_-$, and $\frac{N_-}{c_-}$.
\begin{alignat}{1}
\rho_- &= \frac{\rho_+ N_+}{N_-} = \frac{\rho_+ N_+^2}{c_+^2} \frac{1}{1+z\frac{\gamma-1}{2\gamma}} \\
&= \rho_+ \frac{1+z\frac{\gamma+1}{2\gamma}}{1+z\frac{\gamma-1}{2\gamma}} \myeqlabel{dens-strength}. \\
c_-^2 &= \frac{\gamma p_-}{\rho_-} = \frac{\gamma p_+(1+z)}{\rho_+}\frac{1+z\frac{\gamma-1}{2\gamma}}{1+z\frac{\gamma+1}{2\gamma}} \\
&= c_+^2(1+z)\frac{1+z\frac{\gamma-1}{2\gamma}}{1+z\frac{\gamma+1}{2\gamma}}. \\
N_--c_- &= \frac{c_+^2}{N_+}\left(1+z\frac{\gamma-1}{2\gamma}\right)-c_+ \sqrt{(1+z)\frac{1+z\frac{\gamma-1}{2\gamma}}{1+z\frac{\gamma+1}{2\gamma}}} \\
&= c_+\frac{1+z\frac{\gamma-1}{2\gamma}}{\sqrt{1+z\frac{\gamma+1}{2\gamma}}}-c_+ \sqrt{(1+z)\frac{1+z\frac{\gamma-1}{2\gamma}}{1+z\frac{\gamma+1}{2\gamma}}} \\
&= c_+ \sqrt{\frac{1+z\frac{\gamma-1}{2\gamma}}{1+z\frac{\gamma+1}{2\gamma}}} \left( \sqrt{1+z\frac{\gamma-1}{2\gamma}} - \sqrt{1+z}\right). \\ 
\frac{N_-}{c_-} - 1 &= \frac{\sqrt{1 + z\frac{\gamma-1}{2\gamma}}}{\sqrt{1+z}}-1 \\
\frac{N_-}{c_-} &= \frac{\sqrt{1+z\frac{\gamma-1}{2\gamma}}}{\sqrt{1+z}}
\end{alignat}
We now argue that $z$ is a suitable measure of shock strength.
\begin{lemma}
At a shock, $[|U|]$ and $z$ are equivalent measures of the strength of the shock, i.e. there exists $C>0$ such that
\begin{alignat}{1}
\frac{z}{C} \leq [|U|] \leq Cz
\end{alignat}
\end{lemma}
\begin{proof}
First, recall that we assume pressure is bounded away from $0$ and $\infty$, so we only need estimates valid for
\begin{alignat}{1}
0 \leq z \leq z_{\textrm{max}} := \frac{p_{\textrm{max}}-p_{\textrm{min}}}{p_{\textrm{min}}}.
\end{alignat}
\eqref{dens-strength} shows that
\begin{alignat}{1}
|[\rho]| &= \left| \rho_+ \frac{1+z\frac{\gamma+1}{2\gamma}}{1+z\frac{\gamma-1}{2\gamma}}  - \rho_+ \right| \\
&= \rho_+ \frac{z}{\gamma + z\frac{\gamma-1}{2}}.
\end{alignat}
The function
\begin{alignat}{1}
\frac{z}{\gamma+z\frac{\gamma-1}{2}}
\end{alignat}
is concave for $z \geq 0$, with derivative at $z=0$ equal to $\frac{1}{\gamma}$.  Therefore
\begin{alignat}{1}
\frac{\rho_{\textrm{min}}}{\gamma+z_{\textrm{max}}\frac{\gamma-1}{2}} z \leq |[\rho]| \leq \rho_{\textrm{max}} \frac{1}{\gamma} z.
\end{alignat}
Clearly
\begin{alignat}{1}
p_{\textrm{min}}z \leq |[p]| \leq p_{\textrm{max}}z.
\end{alignat}
Finally, since $-m[N] = [p]$, and
\begin{alignat}{1}
\rho_{\textrm{min}} c_{\textrm{min}} &\leq |m| \leq \rho_{\textrm{max}}|\vec{u}|_{\textrm{max}}, \\
\frac{p_{\textrm{min}}}{\rho_{\textrm{max}}|\vec{u}|_{\textrm{max}}} z &\leq |[N]| \leq \frac{p_{\textrm{max}}}{\rho_{\textrm{min}}c_{\textrm{min}}}z.
\end{alignat}
We have therefore shown that in terms of the primitive variables $V := (\rho, u, v, p)$ that for a shock
\begin{alignat}{1}
\frac{z}{C} \leq |[V]| \leq Cz
\end{alignat}
for some $C > 0$.  $U = (\rho, \rho u, \rho v, \rho E)$ has derivative
\begin{alignat}{1}
U_V = \left( \begin{array}{cccc} 1 & 0 & 0 & 0 \\ u & \rho & 0 & 0 \\ v & 0 & \rho & 0 \\ \frac{u^2+v^2}{2} & \rho u & \rho v & \frac{1}{\gamma-1} \end{array} \right).
\end{alignat}
Therefore it is $C^1$, and the operator norm of its derivative is bounded uniformly on the phase space under consideration.  We can assume our phase space in $V$ variables is convex (if not we can still obtain a Lipschitz estimate from $\int_0^1 U_V(s(V_+ - V_-)+V_-)ds(V_+-V_-)$ since $U(V)$ have derivatives with operator norm bounded away from $0$ and $\infty$ for $0 <C^{-1} \leq \rho \leq C < \infty $) and therefore this is a Lipschitz map.  Furthermore $|[U]| > |[\rho]| > \frac{z}{C},$ and so putting it all together, there must exist $C > 0$ such that
\begin{alignat}{1}
\frac{z}{C} \leq |[U]| \leq Cz
\end{alignat}
for a shock.  We will thus denote $J(U;\theta):= |[U]|$ as the size of the jump in $U$ at $\theta$.
\end{proof}

We now estimate the sizes of the neighborhoods on either side of a shock on which $U$ must be constant.  
\begin{theorem} \mylabel{uniflax}
Suppose $U$ has a forward facing shock at $\theta_0$.  Then the $\sigma^+(\theta_0) > \theta_0$ and $\sigma^-(\theta_0) < \theta_0$ (from Theorem \myref{newthm2}) such that $U$ is constant on\\ $]\sigma^-(\theta_0),\theta_0[,]\theta_0,\sigma^+(\theta_0)[$ satisfy the following:
\begin{alignat}{1}
\sigma^+(\theta_0) &\geq \theta_0 + \delta_L J(U;\theta_0) \\
\sigma^-(\theta_0) &\leq \theta_0 -  \delta_L J(U;\theta_0),
\end{alignat}
where $\delta_L$ is a positive constant independent of $U$.  
Furthermore,
\begin{alignat}{1}
N_+-c_+ &\geq \delta_L J(U;\theta_0), \\
N_- -c_- &\leq -\delta_L J(U;\theta_0).
\end{alignat}
The analogous statement holds for a backward facing shock.
\end{theorem}
\begin{proof}
Suppose the shock is forward facing.  Take 
\begin{alignat}{1}
\sigma^+(\theta_0) := \sup_{\eta > \theta_0} \left\{ \eta \, \Big| \, N(\theta)-c(\theta) > 0  \quad \forall \theta \in ]\theta_0, \eta[ \right\}.
\end{alignat}
The supremum is taken over a non-empty set from Theorem \myref{newthm2}, and is thus defined.  Furthermore, it must be the case that $N(\sigma^+(\theta_0))=c(\theta_0)$ --- from Theorems \myref{newthm2} and \myref{newthm3} there could not be a shock or contact at $\sigma^+(\theta_0)$, and so $U$ is continuous at $\sigma^+(\theta_0)$.  Furthermore, it is constant on $]\theta_0, \sigma^+(\theta_0)[$ by Theorem \myref{newthm1}.  (Note however there could be another forward facing shock at $\theta' > \sigma^+(\theta_0)$ arbitrarily close to $\sigma^+(\theta_0)$).  When $U$ is constant, $\frac{dN}{d\theta} = L$, and so we have
\begin{alignat}{1}
N_+ - c_+ &= N(\theta_0+)-N(\sigma^+(\theta_0)) = \lim_{\eta \searrow \theta_0} \int^\eta_{\sigma^+(\theta_0)} L(\phi)d\phi \\
&\leq (\sigma^+(\theta_0)-\theta_0) |\vec{u}|_{\textrm{max}}.
\end{alignat}
However, recall
\begin{alignat}{1}
N_+ - c_+ = c_+\left(\sqrt{1 + z\frac{\gamma+1}{2\gamma}}-1\right) \geq c_{\textrm{min}} \left(\sqrt{1 + z\frac{\gamma+1}{2\gamma}}-1\right).
\end{alignat}
The function on the right is concave for $z \geq 0$, $0$ at $z=0$, and therefore
\begin{alignat}{1}
N_+ - c_+ \geq c_{\textrm{min}} \frac{\left(\sqrt{1+z_{\textrm{max}}\frac{\gamma+1}{2\gamma}}-1\right)}{z_{\textrm{max}}}z.
\end{alignat}
All together, we obtain
\begin{alignat}{1}
(\sigma^+(\theta_0)-\theta_0) \geq \delta_L' J(U;\theta_0),
\end{alignat}
for some $\delta_L' > 0$.
Similarly, define
\begin{alignat}{1}
\sigma^-(\theta_0) := \inf_{\eta < \theta_0} \left\{ \eta \, \Big| \, 0 < N(\theta) < c(\theta)  \quad \forall \theta \in ]\eta,\theta_0[ \right\}.
\end{alignat}
As before, this infimum is over a non-empty set by Theorem \myref{newthm2}.  Similar to before, it must be the case that either $N(\sigma^-(\theta_0)) = c(\sigma^-(\theta_0))$, or $N(\sigma^-(\theta_0)+) = 0$, since there can be no shocks on $[\sigma^-(\theta_0), \theta_0[$, and no contacts on $]\sigma^-(\theta_0),\theta_0[$.  However, if $N(\sigma^-(\theta_0)+) = 0$, the only possible discontinuity would be a contact, at which $N$ is continuous, and so $N(\sigma^-(\theta_0)+) = N(\sigma^-(\theta_0))=0$ in that case.  Either way, $U$ is continuous on $]\sigma^-(\theta_0),\theta_0[$, hence constant by Theorem \myref{newthm1}, and so $\frac{dN}{d\theta} = L$.  \\
Suppose that $N(\sigma^-(\theta_0)) = c(\sigma^-(\theta_0)) = c_-.$  Then
\begin{alignat}{1}
|N_- - c_-| &= |N(\theta_0-)-N(\sigma^-(\theta_0))| = \Big|\lim_{\eta \nearrow \theta_0} \int_{\sigma^-(\theta_0)}^\eta L(\phi)d\phi \Big| \\
&\leq (\theta_0-\sigma^-(\theta_0)) |\vec{u}|_{\textrm{max}}.
\end{alignat}
Recall
\begin{alignat}{1}
|N_--c_-| &= c_+ \sqrt{\frac{1+z\frac{\gamma-1}{2\gamma}}{1+z\frac{\gamma+1}{2\gamma}}} \left(\sqrt{1+z}- \sqrt{1+z\frac{\gamma-1}{2\gamma}}\right) \\
&\geq c_{\textrm{min}} \sqrt{\frac{\gamma-1}{\gamma+1}} \left(\sqrt{1+z}- \sqrt{1+z\frac{\gamma-1}{2\gamma}}\right) \\
&\geq c_{\textrm{min}}\sqrt{\frac{\gamma-1}{\gamma+1}} \frac{\left(\sqrt{1+z_{\textrm{max}}}-\sqrt{1+z_{\textrm{max}}\frac{\gamma-1}{2\gamma}}\right)}{z_{\textrm{max}}}z,
\end{alignat}
since that function is concave for $z \geq 0$ and $0$ at $z=0$.  So,
\begin{alignat}{1}
(\theta_0-\sigma^-(\theta_0)) \geq \delta_L'' J(U;\theta_0).
\end{alignat}
for some $\delta_L'' > 0$.\\
Suppose instead that $N(\sigma^-(\theta_0))=0$.  Then,
\begin{alignat}{1}
N_- &= N(\theta_0-)-N(\sigma^-(\theta_0)) = \lim_{\eta \nearrow \theta_0} \int_{\sigma^-(\theta_0)}^\eta L(\phi)d\phi \\
&\leq (\theta_0-\sigma^-(\theta_0)) |\vec{u}|_{\textrm{max}}.
\end{alignat}
From before,
\begin{alignat}{1}
N_- = c_- \frac{\sqrt{1+z\frac{\gamma-1}{2\gamma}}}{\sqrt{1+z}} \geq c_{\textrm{min}} \sqrt{\frac{\gamma-1}{2\gamma}}.
\end{alignat}
Therefore,
\begin{alignat}{1}
(\theta_0-\sigma^-(\theta_0)) \geq \delta > 0
\end{alignat}
in this case, for some $\delta >0$.  Then, taking
\begin{alignat}{1}
\delta_L = \min\left(\delta_L', \delta_L'', \frac{\delta}{J_\textrm{max}}\right),
\end{alignat}
we obtain the desired result.  (Recall that there is an upper limit to shock strength for the phase space under consideration).  Analogous calculations yield the statement for backward-facing shocks.
\end{proof}
\section{Decomposition of the domain; regularity}
We now divide $[0, 2\pi)$ into the sets at which different behavior occurs.  We shall assume that there are no stagnation points.
\begin{alignat}{1}
\mathcal{C} &:= \left\{ \theta\, \Big| \, U \mbox{ is continuous at }\theta, |N(\theta)| \neq c(\theta) \mbox{ and } \neq 0 \right\} \\
\mathcal{S}_F &:= \left\{ \theta \, \Big| \,  U \mbox{ has a forward facing shock at } \theta \right\} \\
\mathcal{S}_B &:= \left\{ \theta \, \Big| \,  U \mbox{ has a backward facing shock at } \theta \right\} \\
\mathcal{S}_C &:= \left\{ \theta \, \Big| \,  N(\theta) = 0 \right\} \\
\mathcal{R}_F &:= \left\{ \theta \, \Big| \,  U \mbox{ is continuous at } \theta, N(\theta) = c(\theta) \right\} \\
\mathcal{R}_B &:= \left\{ \theta \, \Big| \,  U \mbox{ is continuous at } \theta, N(\theta) = -c(\theta) \right\}
\end{alignat}
\begin{lemma} \mylabel{unifseparation2}
Recall Lemma \myref{unifseparation}.  In a similar manner, $\mathcal{R}_F$ is uniformly separated from $\mathcal{R}_B \cup \mathcal{S}_B \cup \mathcal{S}_C$, and $\mathcal{R}_B$ is uniformly separated from $\mathcal{R}_F \cup \mathcal{S}_F \cup \mathcal{S}_C$.
\end{lemma}
\begin{proof}
The proof is similar to the proof of Lemma \myref{unifseparation}.
\end{proof}
\begin{lemma} \mylabel{supersubest}
There exists a constant $C_{\mathcal{S}}$ such that if $\theta_0 \in \mathcal{S}_F$, and \\ $\theta \notin ]\sigma^-(\theta_0),\sigma^+(\theta_0)[$, then
\begin{alignat}{1}
J(U;\theta_0), |N(\theta_0+)-c(\theta_0+)|, |N(\theta_0-)-c(\theta_0-)| \leq C_\mathcal{S}|\theta-\theta_0|.
\end{alignat}
Similarly, if $\theta_0 \in \mathcal{S}_B$, and $\theta \notin ]\sigma^-(\theta_0),\sigma^+(\theta_0)[$, then
\begin{alignat}{1}
J(U;\theta_0),|N(\theta_0+)+c(\theta_0+)|,|N(\theta_0-)+c(\theta_0-)| \leq C_\mathcal{S}|\theta-\theta_0|.
\end{alignat}
\end{lemma}
\begin{proof}
Suppose $\theta_0 \in \mathcal{S}_F$.  From Theorem \myref{uniflax}, we have
\begin{alignat}{1}
\delta_L J(U;\theta_0) \leq \min \Big(|\sigma^+(\theta_0)-\theta_0|, |\theta_0-\sigma^-(\theta_0)|\Big) \leq |\theta-\theta_0|.
\end{alignat}
Therefore,
\begin{alignat}{1}
J(U;\theta_0) \leq \delta_L^{-1}|\theta-\theta_0|.
\end{alignat}
From previous calculations, we have that
\begin{alignat}{1}
N_+ - c_+ &= c_+ \left( \sqrt{1+z\frac{\gamma+1}{2\gamma}} -1 \right) \leq c_\textrm{max}\left( \sqrt{1+z\frac{\gamma+1}{2\gamma}} -1 \right).
\end{alignat}
The function on the right is 0 at $z=0$, concave down, and has derivative at $z=0$ equal to $\frac{c_\textrm{max}(\gamma+1)}{4\gamma}$.  Therefore
\begin{alignat}{1}
|N_+-c_+| \leq \frac{c_\textrm{max}(\gamma+1)}{4\gamma}z \leq C_\mathcal{S}' J(U;\theta_0) \leq C_\mathcal{S}\delta_L^{-1}|\theta-\theta_0|.
\end{alignat}
Also recall that
\begin{alignat}{1}
|N_--c_-| &= c_+ \sqrt{\frac{1+z\frac{\gamma-1}{2\gamma}}{1+z\frac{\gamma+1}{2\gamma}}} \left(  \sqrt{1+z}-\sqrt{1+z\frac{\gamma-1}{2\gamma}} \right)\\
&\leq c_\textrm{max}\left(  \sqrt{1+z}-\sqrt{1+z\frac{\gamma-1}{2\gamma}} \right).
\end{alignat}
This function is 0 at $z=0$, concave down, and has derivative at $z=0$ equal to $c_\textrm{max}(\frac{1}{2}-\frac{\gamma-1}{4\gamma})$.  Therefore,
\begin{alignat}{1}
|N_- - c_-| \leq c_\textrm{max} \frac{\gamma+1}{4\gamma}z \leq C_\mathcal{S}'' J(U;\theta_0) \leq C_\mathcal{S}'' \delta_L ^{-1} |\theta-\theta_0|.
\end{alignat}
Taking
\begin{alignat}{1}
C_\mathcal{S} = \max(\delta_L^{-1}, C_\mathcal{S}' \delta_L^{-1}, C_\mathcal{S}'' \delta_L^{-1})
\end{alignat}
gives the desired result.  A similar argument works for $\theta_0 \in \mathcal{S}_B$.
\end{proof}
\begin{lemma} \mylabel{limpoints}
Recall that $\mathcal{S}_C$ is a finite set of points, and therefore can have no limit points.  If $\theta$ is a limit point of $\mathcal{S}_F$, then $\theta \in \mathcal{R}_F$.  If $\theta$ is a limit point of $\mathcal{S}_B$, then $\theta \in \mathcal{R}_B$.
\end{lemma}
\begin{proof}
Consider $\left\{ \theta_n \right\} \searrow \theta$ be a strictly decreasing sequence in $\mathcal{S}_F$.  If $\theta$ is a limit point of $\mathcal{S}_F$, then $\theta \notin ]\sigma^-(\theta_n),\sigma^+(\theta_n)[$ for all $n$.  For all $n$, choose $\theta_n' \in ]\sigma^-(\theta_n),\theta_n[$.  Then,
\begin{alignat}{1}
|N(\theta_n')-c(\theta_n')| &= |N(\theta_n')-c(\theta_n-)|, \mbox{ since $U$ is constant on } ]\sigma^-(\theta_n),\theta_n[,  \notag \\
&= |N(\theta_n')-N(\theta_n-)| + |N(\theta_n-)-c(\theta_n-)| \\
&\leq |\vec{u}|_\textrm{max}|\theta_n'-\theta_n| + C_\mathcal{S}|\theta-\theta_n|, \mbox{ from Lemma \myref{supersubest},} \\
&= \bO(|\theta_n-\theta|).
\end{alignat}
Thus we have a sequence converging to $\theta$ such that 
\begin{alignat}{1}
\lim_{n \rightarrow \infty} |N(\theta'_n)-c(\theta'_n)| = 0.
\end{alignat}
This eliminates the possibility of a shock or contact occurring at $\theta$, and so $U$ is continuous at $\theta$.  Therefore, $\theta \in \mathcal{R}_F$ by definition.  A similar argument works for limit points of $\mathcal{S}_B$.
\end{proof}
\begin{lemma}
If $\theta \in \mathcal{C}$, then $U$ is constant on a neighborhood $]\kappa^-(\theta),\kappa^+(\theta)[$ containing $\theta$.  ($\kappa^\pm(\theta)$ are taken to be maximal so that each is in either $\mathcal{S}_C, \mathcal{S}_F, \mathcal{S}_B, \mathcal{R}_F,$ or $\mathcal{R}_B$.)
\end{lemma}
\begin{proof}
From Lemma \myref{limpoints}, $\theta$ is not a limit point of $\mathcal{S}_C, \mathcal{S}_F, $ or $\mathcal{S}_B$.  Therefore, $U$ is continuous on a neighborhood of $\theta$.  Therefore, if $|N(\theta)| \neq c(\theta)$ or $0$, then this will also be true on a neighborhood of $\theta$.  Then, Theorem \myref{newthm1} applies and so $U$ is constant on some neighborhood containing $\theta$.  Clearly $\kappa^\pm(\theta)$ can be taken to satisfy the requirement in the statement of the lemma.
\end{proof}
\begin{lemma}
Assume there are no stagnation points.  Define, for $\theta \in [0, 2\pi[$, 
\begin{alignat}{1}
U_S(\theta) = \sum_{\phi \in[0,\theta[ \cap (\mathcal{S}_C \cup \mathcal{S}_F \cup \mathcal{S}_B)}(U(\phi+)-U(\phi-)).
\end{alignat}
Then $U_S$ is a right-continuous saltus function (so, by definition, is of bounded variation).
\end{lemma}
\begin{proof}
We have
\begin{alignat}{1}
\sum_{\phi \in (\mathcal{S}_C \cup \mathcal{S}_F \cup \mathcal{S}_B)} |U(\phi+)-U(\phi-)| &=
\sum_{\phi \in ( \mathcal{S}_F \cup \mathcal{S}_B)} J(U;\phi) + \sum_{\phi \in \mathcal{S}_C} |U(\phi+)-U(\phi-)| \\
&\leq (2\delta_L)^{-1} \sum_{\phi \in \mathcal{S}_F \cup \mathcal{S}_B} (\sigma^+(\phi)-\sigma^-(\phi)) + \sum_{\phi \in \mathcal{S}_C} C < \infty,
\end{alignat}
since the number of contacts is finite, the neighborhoods $]\sigma^-(\phi),\sigma^+(\phi)[$ are pairwise disjoint, the phase space is compact, and the domain is compact.  The $BV$ norm only depends on the equation of state and the bounds for the phase space.
\end{proof}
\begin{lemma} \mylabel{uslip}
For any $\theta_0 \in \mathcal{R}_F \cup \mathcal{R}_B$, $U_S$ satisfies a Lipschitz estimate based at $\theta_0$ for $\theta$ sufficiently close to $\theta_0$.  That is,
\begin{alignat}{1}
|U_S(\theta) - U_S(\theta_0)| \leq C_S |\theta-\theta_0|, \quad \forall \theta \mbox{ sufficiently close to } \theta_0.
\end{alignat}
Moreover, the Lipschitz constant $C_S$ is uniform in $\theta$ and independent of $U$.  
\end{lemma}
\begin{proof}
Consider $\theta_0 \in \mathcal{R}_F$, and $\theta > \theta_0$.  We only need to consider other forward facing shocks occurring between $\theta_0$ and $\theta$, since $\mathcal{R}_F$ is uniformly separated from $\mathcal{S}_B$ and $\mathcal{S}_C$.  
Suppose $\theta \notin ]\sigma^-(\phi),\sigma^+(\phi)[$ for any $\phi \in \mathcal{S}_F$.  Then,
\begin{alignat}{1}
|U_S(\theta)-U_S(\theta_0)| &\leq \sum_{\theta_0 < \phi < \theta} J(U;\phi)\\ 
& \leq (2\delta_L)^{-1}\sum_{\theta_0 < \phi < \theta} (\sigma^+(\phi)-\sigma^-(\phi)) \leq (2\delta_L)^{-1} |\theta-\theta_0|,
\end{alignat}
since the $]\sigma^-(\phi),\sigma^+(\phi)[$ are pairwise disjoint and contained in $[\theta_0, \theta]$ by assumption.  If $\theta \in ]\sigma^-(\phi),\phi[$ for some $\phi \in \mathcal{S}_F$, then the previous estimate holds for $\theta = \sigma^-(\phi)$, and $U_S$ is constant on $]\sigma^-(\phi),\phi[$ and so the result follows.  If $\theta \in [\phi, \sigma^+(\phi)[$ for some $\phi \in \mathcal{S}_F$, then apply the previous estimate for $\theta = \sigma^-(\phi)$, and then
\begin{alignat}{1}
|U_S(\theta)-U_S(\theta_0)| &\leq (2 \delta_L)^{-1}|(\sigma^-(\phi)-\theta_0) + J(U;\phi) \\
&\leq ((2 \delta_L)^{-1} + C_\mathcal{S})|\theta-\theta_0|.
\end{alignat}
from Lemma \myref{supersubest}.  Take $C_S = (2\delta_L)^{-1}+C_\mathcal{S}.$  Similar arguments work for $\theta< \theta_0$, and $\theta_0 \in \mathcal{R}_B$.  
\end{proof}
\begin{lemma}
For every $\theta_0 \in \mathcal{R}_F$, there exists a neighborhood containing $\theta_0$ such that 
\begin{alignat}{1} 
u(\theta) \sin \theta_0 - v(\theta) \cos \theta_0 -c(\theta)
\end{alignat}
 satisfies a Lipschitz condition based at $\theta_0$ for all $\theta$ in this neighborhood.  The Lipschitz constant is uniform for all such $\theta_0$ and is independent of $U$.  That is,
\begin{alignat}{1}
\Big|u(\theta) \sin \theta_0 - v(\theta) \cos \theta_0-c(\theta)\Big| \leq M|\theta-\theta_0|, \quad \forall \theta \mbox{ sufficiently close to } \theta_0,
\end{alignat}
(recalling that $(N(\theta_0)-c(\theta_0)) = 0$ for $\theta_0 \in \mathcal{R}_F$). \\
We have the similar estimate for 
\begin{alignat}{1}
u(\theta) \sin \theta_0 - v(\theta) \cos \theta_0+c(\theta) 
\end{alignat}  
with $\theta_0 \in \mathcal{R}_B.$
\end{lemma}
\begin{proof}
Suppose $\theta_0 \in \mathcal{R}_F$.   We first prove the desired estimate for $|N(\theta)-c(\theta)|$. Take $\theta > \theta_0$ sufficiently close to $\theta_0$.  This means that either $\theta \in \mathcal{C}$, with $\kappa^-(\theta) \in \mathcal{R}_F \cup \mathcal{S}_F$, $\theta \in \mathcal{S}_F$, or $\theta \in \mathcal{R}_F$ (from Lemma \myref{unifseparation2}).\\
\\
Suppose $\theta \in \mathcal{R}_F$. \\

Then $|N(\theta)-c(\theta)| = 0$.\\
\\
Suppose $\theta \in \mathcal{S}_F$. \\
Then, recalling
\begin{alignat}{1}
\sigma^-(\theta)=\inf_{\eta < \theta} \left\{ \eta \Big| 0 < N(\phi)-c(\phi) < c(\phi), \quad \forall \phi \in ]\eta, \theta[ \right\},
\end{alignat}
it is clear that $\theta_0 \notin ]\sigma^-(\theta),\theta[$.  Therefore, Lemma \myref{supersubest} applies and
\begin{alignat}{1}
|N(\theta)-c(\theta)| = |N(\theta+)-c(\theta+)| \leq C_\mathcal{S}|\theta-\theta_0|,
\end{alignat}
(recalling we have made $U$ right continuous everywhere).  \\
\\
Finally, suppose $\theta \in \mathcal{C}.$  
If $\kappa^-(\theta) \in \mathcal{R}_F$, then
\begin{alignat}{1}
|N(\theta)-c(\theta)| &= |N(\theta)-c\big(\kappa^-(\theta)\big)| \\
&\leq |N(\theta)-N\big(\kappa^-(\theta)\big)| + |N\big(\kappa^-(\theta)\big)-c\big(\kappa^-(\theta)\big)| \\
&\leq |\vec{u}|_\textrm{max}|\kappa^-(\theta)-\theta| + 0 \\
& \leq |\vec{u}|_\textrm{max}|\theta-\theta_0|.
\end{alignat}
If $\kappa^-(\theta) \in \mathcal{S}_F$, then
\begin{alignat}{1}
|N(\theta)-c(\theta)| &= |N(\theta)-c\big(\kappa^-(\theta)\big)| \\
&= |N(\theta) - N\big(\kappa^-(\theta)\big)| + |N\big(\kappa^-(\theta)\big)-c\big(\kappa^-(\theta)\big)| \\
&\leq |\vec{u}|_\textrm{max}|\theta-\kappa^-(\theta)| + C_\mathcal{S}|\theta-\theta_0| \\
&\ \leq |\vec{u}|_\textrm{max}|\theta-\theta_0| + C_\mathcal{S}|\theta-\theta_0|.
\end{alignat}
Taking $M' = \max(C_\mathcal{S}, |\vec{u}|_\textrm{max})$ gives the desired estimate for $|N(\theta)-c(\theta)|$.  Then,
\begin{alignat}{1}
|u(\theta) \sin \theta_0 - v(\theta) \cos \theta_0 -c(\theta)| &= |N(\theta)-c(\theta)| + |u(\theta)\big(\sin \theta_0 - \sin \theta \big) - v(\theta) \big(\cos \theta_0 - \cos \theta \big)| \\
&\leq M' |\theta-\theta_0| + 2 |\vec{u}|_\textrm{max} |\theta-\theta_0|.
\end{alignat}
Taking $M := M' + 2 |\vec{u}|_\textrm{max}$ gives the desired result for $\theta > \theta_0$, $\theta$ sufficiently close to $\theta_0 \in \mathcal{R}_F$.\\

  Similar arguments work for $\theta < \theta_0$, and for $\theta_0 \in \mathcal{R}_B$. 
\end{proof}
We now recall the concept of genuine nonlinearity.  For fixed $\theta$, the quantities $N \pm c$ are genuinely nonlinear in the sense that
\begin{alignat}{1}
(N \pm c)_U r^\pm(U;\theta) \neq 0
\end{alignat}
for all $U$.  For simplicity we can compute the derivative in terms of the primitive variables, and recalling $c = \sqrt{\gamma\frac{p}{\rho}}$ we obtain
\begin{alignat}{1}
&(N \pm c)_U r^\pm(U;\theta) = \left(N \pm \sqrt{\gamma\frac{p}{\rho}}\right)_V V_U r^\pm(U;\theta) \\
&= \left( \begin{array}{cccc} \mp \frac{1}{2}\sqrt{\gamma\frac{p}{\rho^3}} & \sin \theta & -\cos \theta & \pm \frac{1}{2}\sqrt{\frac{\gamma}{p\rho}} \end{array} \right) \left( \begin{array}{cccc} 1 & 0 & 0 & 0 \\ u & \rho & 0 & 0 \\ v & 0 & \rho & 0 \\ \frac{u^2+v^2}{2} & \rho u & \rho v & \frac{1}{\gamma-1} \end{array}\right)^{-1} \left( \begin{array}{c} 1 \\ u \pm c\sin \theta  \\ v \mp c \cos \theta  \\ H \pm Nc \end{array}\right)  \notag \\
&=\left( \begin{array}{cccc} \mp \frac{1}{2}\frac{c}{\rho} & \sin \theta & -\cos \theta & \pm \frac{1}{2}\frac{c}{p} \end{array} \right) \left( \begin{array}{cccc} 1 & 0 & 0 & 0 \\ \frac{u}{\rho} & \frac{1}{\rho} & 0 & 0 \\ -\frac{v}{\rho} & 0 & \frac{1}{\rho} & 0 \\ (\gamma-1)\frac{u^2+v^2}{2} & -(\gamma-1)u & -(\gamma-1) v & \gamma-1 \end{array}\right) \left( \begin{array}{c} 1 \\ u \pm c\sin \theta  \\ v \mp c \cos \theta  \\ H \pm Nc \end{array}\right)  \notag \\
&= \left( \begin{array}{cccc} \mp \frac{1}{2}\frac{c}{\rho} & \sin \theta & -\cos \theta & \pm \frac{1}{2}\frac{c}{p} \end{array} \right) \left( \begin{array}{c} 1 \\ \pm \frac{c}{\rho} \sin\theta \\ \mp \frac{c}{\rho} \cos \theta \\ c^2 \end{array}\right) \notag \\
&= \mp \frac{1}{2} \frac{c}{\rho} \pm \frac{c}{\rho} \pm \frac{1}{2} \frac{c^3}{p} \\
&= \mp \frac{1}{2} \frac{c}{\rho} \pm\frac{c}{\rho} \pm \frac{\gamma}{2}\frac{c}{\rho} \\
&= \pm \frac{1}{2} (\gamma + 1) \frac{c}{\rho} \neq 0.
\end{alignat}

The Roe linearization for the full polytropic Euler equations has the advantage that it is simply the matrices $f^x_U$ and $f^y_U$ evaluated at some appropriately averaged state $\Ub$.  It takes the form (see \cite{tran})
\begin{alignat}{1}
&\hat{A}(U_-,U_+;\theta) := A(\Ub;\theta) = \\
 &\left(\begin{array}{cccc} 0 & \sin \theta & -\cos \theta & 0 \\ \sin \theta \frac{\gamma-1}{2}(\overline{u}^2+\overline{v}^2) - \overline{u}\overline{N} & \overline{N}+(2-\gamma) \overline{u}\sin \theta & -\overline{u}\cos \theta - (\gamma-1)\overline{v}\sin \theta & \sin\theta(\gamma-1) \\ -\cos\theta \frac{\gamma-1}{2}(\overline{u}^2+\overline{v}^2) - \overline{v}\overline{N} & \overline{v}\sin \theta + (\gamma-1)\overline{u}\cos \theta & \overline{N}  +(\gamma-2)\overline{v}\cos \theta & -\cos \theta(\gamma-1) \\ \left(\frac{\gamma-1}{2}(\overline{u}^2+\overline{v}^2) - \overline{h}\right)\overline{N} & \overline{h} \sin \theta - (\gamma-1)\overline{N}\overline{u} & -\overline{h}\cos \theta -(\gamma-1)\overline{N}\overline{v} & \gamma \overline{N} \end{array}\right). \notag
\end{alignat}
\vspace{1pc}
We define
\begin{alignat}{1}
h := e + \frac{1}{2}(u^2+v^2) + \frac{p}{\rho} = \frac{\gamma}{\gamma-1}\frac{p}{\rho}+\frac{1}{2}(u^2+v^2).
\end{alignat}
to be the total specific enthalpy per unit mass, so that
\begin{alignat}{1}
c^2 = (\gamma-1)\left(h-\frac{u^2+v^2}{2}\right)
\end{alignat}
is the sound speed.  The averaged quantities are defined as
\begin{alignat}{1}
\overline{u} &= \frac{u_-\sqrt{\rho_-}+u_+\sqrt{\rho_+}}{\sqrt{\rho_-}+\sqrt{\rho_+}},\\
\overline{v} &= \frac{v_-\sqrt{\rho_-}+v_+\sqrt{\rho_+}}{\sqrt{\rho_-}+\sqrt{\rho_+}}, \\
\overline{N} &= \overline{u} \sin \theta - \overline{v} \cos \theta, \\
\overline{h} &= \frac{h_- \sqrt{\rho_-}+h_+\sqrt{\rho_+}}{\sqrt{\rho_-}+\sqrt{\rho_+}}, \\
\overline{c}^2 &= (\gamma-1)\left(\overline{h}-\frac{\overline{u}^2+\overline{v}^2}{2}\right).
\end{alignat}
The eigenvalues of $A(\Ub;\theta)$ are
\begin{alignat}{1}
\overline{N} \pm \overline{c}, \overline{N}, \overline{N},
\end{alignat}
and it has a full basis of eigenvectors.  Moreover, it is clear that $A(\Ub;\theta)$ is a smooth function of $U_\pm$, the eigenvalues are smooth functions of $U_\pm$, (away from $\rho = 0$ of course), and by direct inspection of the eigenvectors (they are not needed here, but expressions for them are available) they too are smooth functions of $U_\pm$.  
Define the left and right eigenvectors so that
\begin{alignat}{1}
A(\Ub;\theta)r^\pm(\Ub;\theta) &= (\overline{N}\pm\overline{c})r^\pm(\Ub;\theta), \\
 l^\pm(\Ub;\theta)A(\Ub;\theta) &= (\overline{N}\pm \overline{c})l^\pm(\Ub;\theta), \\
A(\Ub;\theta)r^i(\Ub;\theta) &= \overline{N}\, r^i(\Ub;\theta) \mbox{ for } i = 1,2, \\
 l^i(\Ub;\theta)A(\Ub;\theta) &= \overline{N}\, l^i(\Ub;\theta), \mbox{ for } i = 1,2. \\
l^\alpha(\Ub;\theta) r^\alpha(\Ub;\theta) &= \delta_{\alpha\beta}, \quad \alpha, \beta = +, -, 1, 2.
\end{alignat}

\begin{theorem} \mylabel{ulip}
For any $\theta_0 \in \mathcal{R}_F \cup \mathcal{R}_B$, there is a neighborhood containing $\theta_0$ such that $U$ satisfies a Lipschitz condition based at $\theta_0$ for all $\theta$ in this neighborhood.  The Lipschitz constant is uniform for all such $\theta_0$ and is independent of $U$.  That is,
\begin{alignat}{1}
\Big|U(\theta)-U(\theta_0)\Big| \leq M' |\theta-\theta_0|, \forall \theta \mbox{ sufficiently close to } \theta_0.
\end{alignat}
\end{theorem}
\begin{proof}
Suppose $\theta_0 \in \mathcal{R}_F$.  \\
Recall that
\begin{alignat}{1}
\hat{A}\big(U(\theta_0),U(\theta);\theta_0\big)\big(U(\theta)-U(\theta_0)\big) = \bO(|\theta-\theta_0|).
\end{alignat}
Denote $l^+\big(U(\theta_0),U(\theta);\theta_0\big) = l^+(\Ub;\theta_0)$ where the average is taken between $U(\theta_0)$ and $U(\theta)$.  Left multiply by $l^+\big(U(\theta_0),U(\theta);\theta_0\big)$ to obtain
\begin{alignat}{1}
\big(\overline{N}+\overline{c}\big) l^+\big(U(\theta_0),U(\theta);\theta_0\big) \big(U(\theta)-U(\theta_0)\big) = \bO(|\theta-\theta_0|),
\end{alignat}
where the averages are taken between $U(\theta_0)$ and $U(\theta)$.  Since $\theta_0 \in \mathcal{R}_F$, $U$ is continuous at $\theta_0$.  Therefore for $\theta$ sufficiently close to $\theta_0$, $\overline{N}+\overline{c}$ is uniformly bounded away from zero, since $\overline{N}-\overline{c}$ will be approaching zero.  Therefore,
\begin{alignat}{1}
\theta \mapsto l^+\big(U(\theta_0),U(\theta_1);\theta_0\big)\big(U(\theta)-U(\theta_0)\big),
\end{alignat}
satisfies a Lipschitz estimate based at $\theta_0$ for $\theta$ sufficiently close to $\theta_0$, with Lipschitz constant uniformly bounded above by the bounds on the phase space, and proportional $c_\textrm{min}^{-1}$.  Similarly,
\begin{alignat}{1}
\theta \mapsto l^i\big(U(\theta_0),U(\theta);\theta_0\big)\big(U(\theta)-U(\theta_0)\big), i = 1,2,
\end{alignat}  
also satisfies a Lipschitz estimate based at $\theta_0$ for $\theta$ sufficiently close to $\theta_0$, with a similar upper bound on the Lipschitz constant.  We claim that
\begin{alignat}{1}
W \mapsto g(W)  = \left(\begin{array}{c} g^1(W) \\ g^2(W) \\ g^3(W) \\ g^4(W) \end{array} \right):= \left( \begin{array}{c} \dfrac{W_2}{W_1} \sin \theta_0 - \dfrac{W_3}{W_1} \cos \theta_0 - c(W) \\ l^+\big(U(\theta_0),W ; \theta_0\big)\big(W-U(\theta_0)\big) \\ l^1\big(U(\theta_0),W;\theta_0\big)\big(W-U(\theta_0)\big) \\ l^2\big(U(\theta_0),W;\theta_0\big)\big(W-U(\theta_0)\big) \end{array} \right)
\end{alignat}
defines a diffeomorphism for $W$ sufficiently close to $U(\theta_0)$, when combined with the previous lemma will prove the claim.  Notice that
\begin{alignat}{1}
g_W\big(U(\theta_0)\big) = \left( \begin{array}{c} (u \sin \theta_0 - v \cos \theta_0 - c)_U \big|_{U(\theta_0)} \\ l^+\big(U(\theta_0);\theta_0\big) \\ l^1\big(U(\theta_0);\theta_0\big) \\l^2\big(U(\theta_0);\theta_0\big) \end{array} \right).
\end{alignat}
Then, if
\begin{alignat}{1}
g^i_W\big(U(\theta_0)\big)z = 0, i = 2,3,4,
\end{alignat}
 this implies $z \parallel r^-\big(U(\theta_0);\theta_0)$.  But since
\begin{alignat}{1}
\Big((u \sin \theta_0 - v \cos \theta_0 -c)_U r^-(U;\theta_0)\Big)\Big|_{U(\theta_0)} \neq 0,
\end{alignat}
(by genuine nonlinearity) this implies $z=0$.  Therefore, for $W$ sufficiently close to $U(\theta_0)$, $g$ is a diffeomorphism.  Since $U(\theta)$ approaches $U(\theta_0)$ as $\theta \rightarrow \theta_0$ by continuity, and $g\big(U(\theta)\big)$ satisfies the Lipschitz estimate based at $\theta_0$ for $\theta$ sufficiently close to $\theta_0$, $U$ itself must satisfy a Lipschitz estimate based at $\theta_0$ for $\theta$ sufficiently close to $\theta_0$.  The $C^1$ norm of $g$ is bounded uniformly above by the phase space bounds and equation of state, as is the Lipschitz constant for $g\big(U(\theta)\big)$,  and so the Lipschitz constant for $U$ is as well.
\end{proof}
\begin{theorem} \mylabel{lipschitz}
Assuming that there are no stagnation points, and that density and internal energy remain bounded away from zero, any $L^\infty$ weak, steady, self similar solution to the 2-d full polytropic Euler equations must be of bounded variation.  Moreover, $U$ can be decomposed as
\begin{alignat}{1}
U = U_L + U_S,
\end{alignat}
where $U_L$ is Lipschitz with constant independent of $U$, and $U_S$ is a saltus function of bounded variation, with total variation independent of $U$.  (Note these constants will depend on the equation of state, the lower bound on density and internal energy, and the $L^\infty$ norm of $U$.)  Note that this implies $U$ is a special function of bounded variation, since the Cantor part vanishes.  Moreover, the absolutely continuous part is in fact Lipschitz.
\end{theorem}
\begin{proof}
The statement about $U_S$ has been covered in previous lemmas.  We claim that for any $\theta_0$, $U_L$ satisfies a Lipschitz estimate based at $\theta_0$ for $\theta$ sufficiently close to $\theta_0$.  \\
\\
If $\theta_0 \in \mathcal{C}$, then $U$ is constant on a neighborhood containing $\theta_0$, and since there are no shocks or contacts it is clear that $U_L := U-U_S$ is constant and thus satisfies a Lipschitz estimate based at $\theta_0$ with constant 0. \\
\\
If $\theta_0 \in \mathcal{S}_F \cup \mathcal{S}_B \cup \mathcal{S}_C$, then the jump at $\theta_0$ is accounted for in $U_S$, and so $U_L$ is constant on some neighborhood containing $\theta_0$, and so satisfies a Lipschitz estimate based at $\theta_0$ with constant 0.\\
\\
If $\theta_0 \in \mathcal{R}_F \cup \mathcal{R}_B$, then, for $\theta$ sufficiently close to $\theta_0$, we have from Lemma \myref{uslip} and Theorem \myref{ulip} that
\begin{alignat}{1}
|U_L(\theta)-U_L(\theta_0)| &\leq |U(\theta)-U(\theta_0)| + |U_S(\theta)-U_S(\theta_0)| \\
&\leq C_S|\theta-\theta_0| + M'|\theta-\theta_0| := C_L |\theta-\theta_0|.
\end{alignat}
Pick any $\theta_1$ and $\theta_2$.  Consider the open cover
\begin{alignat}{1}
\bigcup_{\phi \in [\theta_1,\theta_2]} \Omega(\phi),
\end{alignat}
where for any $\theta$, $\Omega(\theta)$ is the neighborhood for which we have a Lipschitz estimate based at $\theta$ with Lipschitz constant uniformly bounded by $C_L$.  This has a finite subcover
\begin{alignat}{1}
\bigcup_{n=1}^N \Omega(\phi_n).
\end{alignat}
Then, adding in $\Omega(\theta_1)$ and $\Omega(\theta_2)$, we can express
\begin{alignat}{1}
|U(\theta_2)-U(\theta_1)| &\leq |U(\eta_1) - U(\theta_1)| + |U(\phi_1)-U(\eta_1)| + |U(\eta_2)-U(\phi_1)|  \\
& |U(\phi_2)-U(\eta_2)| + ... + |U(\theta_2)-U(\eta_{N+1})| \\
\leq C_L|\theta_2-\theta_1|,
\end{alignat}
where $\phi_{i-1} < \eta_i < \phi_i$, and $\eta_i \in \Omega(\phi_{i-1}) \cap \Omega(\phi_i)$ for $i=2,..,N$.  (We take $\theta_1 < \eta_1 < \phi_1, \eta_1 \in \Omega(\theta_1) \cap \Omega(\phi_1)$, and $\phi_N < \eta_{N+1} < \theta_2,$ with $\eta_{N+1} \in \Omega(\phi_N) \cap \Omega(\theta_2)$.)  Thus $U_L$ is Lipschitz on all of $[0,2\pi)$, and the rest follows.
\end{proof}

\section{Structure of flows}
We now prove some results about the structure of possible solutions, and present several examples.  \\

We begin by decomposing the domain into a finite number of sectors.  Denote the points in $\mathcal{S}_C$ as $\theta_1, \theta_2,...,\theta_N$ so that

\begin{alignat}{1}
0 \leq \theta_1 < \theta_2 < ...< \theta_N < 2\pi.
\end{alignat}

(Recall that $\theta \in \mathcal{S}_C$ means that $N(\theta)=0$, and $\mathcal{S}_C$ is a finite set by Lemma \myref{finitecontact}.)  Define the sectors $I_i$, for $i=1,..,N$, as
\begin{alignat}{1}
I_i := [\theta_i, \theta_{i+1}],
\end{alignat}
(taking $\theta_{N+1} = \theta_1$ to unify the notation).\\

We say that $I_i$ is a \emph{forward sector} if $N|_{I_i} \geq 0$, and that $I_i$ is a \emph{backward sector} if $N|_{I_i} \leq 0$.  By construction, each sector will either be forward or backward, and $N$ will be positive on the interior of a forward sector, and negative on the interior of a backward sector.  Moreover, $L$ is continuous on the interior of each sector, since $L$ is continuous at shocks and there are no contacts in the interior of a sector.\\

For all the figures in the remainder of the paper, the flow direction is what is indicated.  The length of the arrows is not meant to suggest anything about the length of the velocity vectors.

\begin{figure}[h]
  \centering
  \def\svgwidth{250pt}
  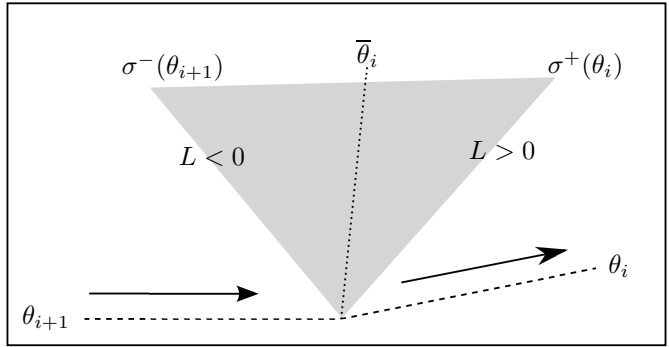
  \caption{In a forward sector $I_i = [\theta_i, \theta_{i+1}]$, $L(\theta_i+) > 0$ and $L(\theta_{i+1}-) < 0$.  $L$ is monotonically decreasing on $]\theta_i, \theta_{i+1}[$, and equal to zero at a unique $\overline{\theta}_i$.  The flow is constant on $]\theta_i, \sigma^+(\theta_i)[$ and $]\sigma^-(\theta_{i+1}), \theta_{i+1}[$.  If $\theta_{i+1} \neq \theta_i + \pi$, then there must be some wave structure in the grey shaded region.}
   \label{tangpropfig}
\end{figure}

\begin{lemma} \mylabel{tangprop}
(See Figure \ref{tangpropfig}.) Suppose $I_i$ is a forward sector.  Then, $L(\theta_{i}+) > 0$, $L(\theta_{i+1}-) < 0$, and $L$ is monotonically decreasing on $]\theta_i, \theta_{i+1}[$.  Similarly, if $I_i$ is a backward sector then $L(\theta_i+) < 0$, $L(\theta_{i+1}-) > 0$, and $L$ is monotonically increasing on $]\theta_i, \theta_{i+1}[$.  Therefore, there exists a unique $\overline{\theta}_i \in \, ]\theta_i,\theta_{i+1}[$ such that $L(\overline{\theta}_i) = 0$.  
\end{lemma}
\begin{proof}
Consider the strong form of the conservation of mass and tangential momentum equations,
\begin{alignat}{1}
(\rho N)_\theta &= \rho L,\\
(\rho L N)_\theta &= \rho L^2 - \rho N^2.
\end{alignat}
Manipulating these, we obtain for any point of differentiability on the interior of $I_i$ that
\begin{alignat}{1}
L_\theta \rho N + L (\rho N)_\theta &= \rho L^2 - \rho N^2, \\
L_\theta \rho N + \rho L^2 &= \rho L^2 - \rho N^2 \\
L_\theta &= -N, \myeqlabel{tangderiv}
\end{alignat}
since $\rho$ is bounded away from zero by assumption, and $N\neq 0$ on the interior of $I_i$.  \\

Theorem \myref{lipschitz} shows that $U$ is Lipschitz almost everywhere (since $U_S$ is constant except on at most a countable, discrete set), and since jumps in $U_S$ on the interior of $I_i$ must be shocks (not contacts), $L_S$ is constant on the interior of $I_i$.  Therefore, $L$ is Lipschitz (hence differentiable almost everywhere) on the interior of $I_i$, and so the fundamental theorem of calculus can be applied to $L$.  Therefore \eqref{tangderiv} shows that $L$ is monotonically decreasing (increasing) on $]\theta_i, \theta_{i+1}[$ if $I_i$ is a forward (backward) sector.\\

Recall that when $U$ is constant, $N_\theta = L$.  Also recall that by Theorem \myref{newthm3} there exist $\sigma^+(\theta_i) > \theta_i$ and $\sigma^-(\theta_{i+1})< \theta_{i+1}$ such that $U$ is constant on $]\theta_i, \sigma^+(\theta_i)[$ and on $]\sigma^-(\theta_{i+1}), \theta_{i+1}[$.  Therefore, the following right limits are defined and we have for small $\delta > 0$ that
\begin{alignat}{1}
N(\theta_i + \delta) - N(\theta_i+) = N(\theta_i+\delta) =  \int_{\theta_i+}^{\theta_i+\delta} L(\eta) d\eta.
\end{alignat} 
Since $U$ is constant on $]\theta_i, \theta_i+\delta[$ and there are no stagnation points $0 \neq |\vec{u}(\theta_i +)| = |N(\theta_i+)| + |L(\theta_i+)| = |L(\theta_i+)|$.  Therefore, by continuity, $\sgn(L)$ is constant on $]\theta_i, \theta_{i}+\delta[$, and so $\sgn\big(N(\theta_i+\delta)\big) = \sgn\big(L(\theta_i+)\big)$.  Therefore, for a forward sector, $L(\theta_i+) > 0$.  Similar arguments work for $L(\theta_{i+1}-)$ and for backward sectors.  \\
Since $L(\theta_i+)$ and $L(\theta_{i+1}-)$ must have opposite signs, and $L$ is monotonic on $]\theta_i, \theta_{i+1}[$, there is a unique $\overline{\theta}_i$ such that $L(\overline{\theta}_i) = 0$.
\end{proof}

Now we define a \emph{Prandtl-Meyer wave}.  A \emph{forward Prandtl-Meyer wave} is a closed interval $[\alpha, \beta]$ such that $N(\theta) = c(\theta)$ for all $\theta \in [\alpha, \beta]$.  A \emph{backward Prandtl-Meyer wave} is the same except that $N(\theta) = -c(\theta)$.  Moreover, $U$ is differentiable almost everywhere on $]\alpha, \beta[$, and $U_\theta$ is in the kernel of $A(U;\theta)$ when it is defined. This follows from the the strong form of the Euler equations
\begin{alignat}{1}
A(U;\theta)U_\theta = 0.
\end{alignat}

It is well known that $p = A(S) \rho^{\gamma}$ for a polytropic gas, where $A(S)$ is a function of the entropy, and that away from discontinuities $S$ is constant.  Therefore, for $\theta$ in the interior of a forward Prandtl-meyer wave we have that
\begin{alignat}{1}
(\rho N)_\theta = (\rho c)_\theta &= \sqrt{A(S)\gamma} \big(\rho^{\frac{\gamma+1}{2}} \big)_\theta \\
&= \sqrt{A(s)\gamma} \frac{\gamma+1}{2} \rho^{\frac{\gamma-1}{2}} \rho_\theta = \rho L,
\end{alignat}
and so
\begin{alignat}{1}
\sgn(\rho_\theta) = \sgn(L).
\end{alignat}
Therefore, for a forward Prandtl-Meyer wave, as the gas particles pass through it (corresponding to the decreasing $\theta$ direction by our choice of coordinates), $\rho$ increases if $L$ is negative, and decreases if $L$ is positive.  Since the flow is isentropic,
\begin{alignat}{1}
\sgn(p_\theta) = \sgn(\rho_\theta) = \sgn(L).
\end{alignat}
Similar calculations for backward waves can be done.  Therefore, in light of Lemma \myref{tangprop}, we have the classifications:
\begin{itemize}
\item forward expansion wave: $\alpha < \beta \leq \overline{\theta_i}$, $N(\theta)=c(\theta), L(\theta) \geq 0$ for all $\theta \in [\alpha, \beta]$,
\item forward compression wave: $ \overline{\theta_i} \leq \alpha < \beta $, $N(\theta)=c(\theta), L(\theta) \leq 0$ for all $\theta \in [\alpha, \beta]$,
\item backward expansion wave: $ \overline{\theta}_i \leq \alpha < \beta$, $N(\theta)=-c(\theta), L(\theta) \geq 0$ for all $\theta \in [\alpha, \beta]$,
\item backward compression wave: $ \alpha < \beta \leq \overline{\theta}_i $, $N(\theta)=-c(\theta), L(\theta) \leq 0$ for all $\theta \in [\alpha, \beta]$.
\end{itemize}
It is possible to join a forward compression wave between $[\alpha_1, \beta_1]$ to a forward expansion wave between $[\alpha_0, \beta_0]$ if $\beta_0 = \overline{\theta}_i = \alpha_1$.  In that case the compression wave ends when the flow is precisely sonic at $\overline{\theta}_i$, and an expansion wave immediately starts at $\overline{\theta}_i$.  We have the following theorem.

\begin{figure}[h] 
  \centering
  \def\svgwidth{250pt}
  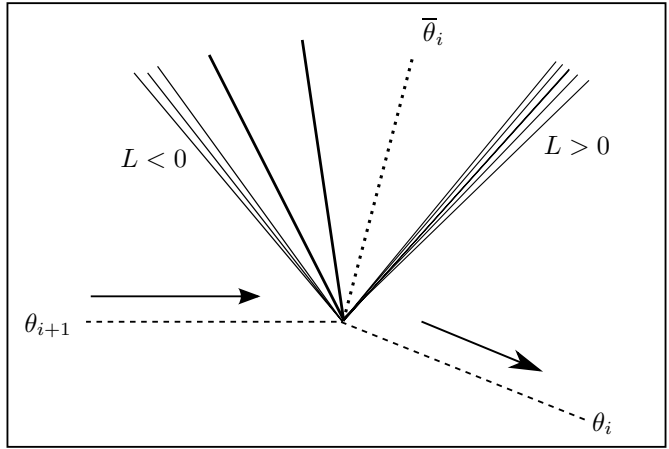
  \caption{In a forward sector $I_i = [\theta_i, \theta_{i+1}]$, $L > 0$ on $]\theta_i, \overline{\theta}_i[$, and $L < 0$ on $]\overline{\theta}_i, \theta_{i+1}[$.  The flow is constant on $]\theta_i, \sigma^+(\theta_i)[$ and $]\sigma^-(\theta_{i+1}), \theta_{i+1}[$.  There is at most one shock or rarefaction in $[\sigma^+(\theta_i),\overline{\theta}_i]$, and possibly infinitely many shocks and compression waves in $]\overline{\theta_i}, \sigma^-(\theta_{i+1})]$.  However, there cannot be consecutive compression waves.  In this particular example, the flow consists of a compression wave and two shocks in the $L<0$ part, and a rarefaction wave in the $L>0$ part.}
  \mylabel{wavespic}
\end{figure}

\begin{theorem} \mylabel{onewave}
(See Figure \myref{wavespic}.) Suppose $I_i$ is a forward sector, and that $U$ is continuous on an open interval  $B \subset ]\overline{\theta}_i, \theta_{i+1}[$. (In this case $L < 0$ on $B$.)  Then, either $U$ is constant on this open interval or constant on either side of a single forward compression wave. \\

Similarly, in a backward sector, on an open interval $B \in ]\theta_{i}, \overline{\theta}_i[$ on which $U$ is continuous, $U$ must be constant or constant on either side of a single backward compression wave.
\end{theorem}
\begin{proof}
Suppose $I_i$ is a forward sector.  Since $U$ is continuous on $B$, the set on which $N(\theta)=c(\theta)$ is closed in $B$.  Therefore, its complement in $B$ is a countable union of open intervals, on which $U$ is constant by Theorem \myref{newthm1} (since $N(\theta)$ cannot be $0$ or $-c(\theta)$ in the interior of $I_i$).  Since $L(\theta)$ is negative on $B$, and $N_\theta = L$ on this complement, $N(\theta) = c(\theta)$ can be satisfied at at most one endpoint of an open interval in the complement.  Therefore at least one endpoint must be an endpoint of $B$, making $\mathcal{R}_F \cap B$ a closed interval in $B$.  \\

Theroem \myref{lipschitz} shows that $U$ is Lipschitz on $B$, since $B \subset \mathcal{C} \cup \mathcal{R}_F$ and thus $U_S$ is constant on $B$.  Therefore, $U$ is differentiable almost everywhere in $\mathcal{R}_F \cap B$ and the strong form of the equations implies $U_\theta$ is in the kernel of $A(U;\theta)$ everywhere it is defined, and so $\mathcal{R}_F \cap B$ defines a forward Prandtl-Meyer wave.  The fact that $B \subset ]\overline{\theta}_i, \theta_{i+1}[$ shows it must be a forward compression wave.  The argument for a backward sector is similar.
\end{proof}
Note that there may be multiple forward compression waves in a forward sector - this theorem requires only that there is at least one forward facing shock in between.  Since $L$ is negative on $]\overline{\theta}_i, \theta_{i+1}[$, on any interval on which $U$ is constant $N$ is decreasing.  In a forward sector, this corresponds to $N$ increasing along particle paths of the gas particles.  Therefore, upon exiting a compression wave, the normal velocity is sonic, but as the gas particles continue traveling in the negative $\theta$ direction, the normal velocity increases and becomes supersonic, leading to the possibility of a forward facing shock, which upon exit the normal velocity will be subsonic.  Normal velocity can then increase along particle paths back to the sound speed, and the gas can enter another compression wave.
\begin{theorem} \mylabel{Lpos}
(See Figure \myref{wavespic}.) Suppose $I_i$ is a forward sector.  Then, on $]\theta_i, \overline{\theta}_i]$, exactly one of the following is true:
\begin{itemize}
\item $U$ is constant on either side of a forward facing shock,
\item $U$ is constant on either side of a forward expansion wave,
\item $U$ has an expansion wave on $[\alpha, \overline{\theta_i}]$ and is constant on $]\theta_i, \alpha[$,
\item $U$ has a normal shock (that is, L = 0) at $\overline{\theta}_i$ and is constant on $]\theta_i, \overline{\theta}_i[$ ,
\item $U$ is constant on $]\theta_i, \overline{\theta}_i]$.  
\end{itemize}
We have the similar statement if $I_i$ is a backward sector, for the interval $[\overline{\theta}_i, \theta_{i+1}[$. 
\end{theorem}
\begin{proof}
Suppose $U$ has a shock at $\theta_0 \in ]\theta_i, \overline{\theta}_i]$.  Then, we know that $N(\theta_0 -) < c(\theta_0-)$.  Recall that there exists $\sigma^-(\theta_0) < \theta_0$ such that $U$ is constant on $]\sigma^-(\theta_0), \theta_0[$, and that $N\big(\sigma^-(\theta_0)+\big) = 0$ or $c\big(\sigma^-(\theta_0)+\big)$.  However, $N_\theta = L$ on $]\sigma^-(\theta_0), \theta[$, and on this interval $L > 0$, and so $N$ decreases as $\theta$ decreases, and so $N\big(\sigma^-(\theta_0)+\big) = 0$, making $\sigma^-(\theta_0) = \theta_1$.  Therefore, there can be no shocks in $]\theta_i, \theta_0[$, and $U$ is constant on $]\theta_i, \theta_0[$ by Theorem \myref{newthm1}.  Similar arguments show that $\sigma^+(\theta_0) \geq \overline{\theta}_i$ (since for $N$ to be sonic it must decrease from $N(\theta_0) > c(\theta_0)$, which is impossible since $N_\theta = L > 0$ on $]\theta_0, \sigma^+(\theta_0)[$) and so either the first or fourth statement is true.\\

If there is not a shock, then $U$ is continuous on this interval, and similar arguments as in the proof of Theorem \myref{onewave} show that there can be at most one expansion wave, and we are done.  Similar arguments work in backward sectors.
\end{proof}
Examples with infinitely many shocks can be constructed (these theorems show that they must occur in the parts of the sectors where $L<0$), or with infinitely many shocks interspersed with compression waves (with the restriction that compression waves cannot occur consecutively, by Theorem \myref{onewave}).  Therefore, since infinitely many discontinuities may occur $BV$ is the sharpest commonly used function space we may use.

\section{Maximum number of contacts}
We note for a both rarefaction and compression waves that the velocity turns towards the origin as the gas particles travel through the wave.  This can be seen by manipulating
\begin{alignat}{1}
(\rho N)_\theta &= \rho L,\\
(\rho N^2 + p)_\theta &= 2\rho N L,
\end{alignat}
to obtain the following.
\begin{alignat}{1}
(\rho N)_\theta N + (\rho N)N_\theta + p_\theta &= 2 \rho N L,\\
\rho L N + (\rho N)N_\theta + p_\theta &= 2 \rho N L, \\
N N_\theta - LN &= -\frac{p_\theta}{\rho},\\
LN - N N_\theta &= \frac{p_\theta}{\rho}.
\end{alignat}
We consider the angle of the flow, $\phi := \angle(u,v)$, as in Figure \myref{anglefig} , as in \cite{LL}.

\begin{figure}[h] 
  \centering
  \vspace{-1cm}
  \def\svgwidth{180pt}
  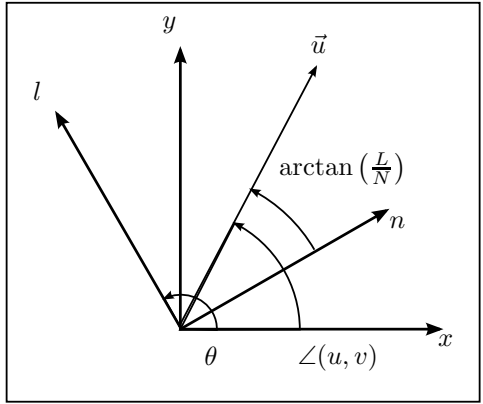
  \vspace{-1cm}
  \caption{Computing the angle of the flow in terms of $N, L,$ and $\theta$.  $n$ is the angular coordinate vector, and $l$ is the radial coordinate vector.}
  \mylabel{anglefig}
\end{figure}

Since we are considering $N>0$, $\angle(N,L) \in (-\pi/2, \pi/2)$, and so $\angle(N,L) = \arctan\left(\frac{L}{N}\right)$.  Then,
\begin{alignat}{1}
\angle(u,v) &= \theta - \frac{\pi}{2} + \arctan\left(\frac{L}{N}\right) \\
 &= \theta - \arctan\left(\frac{N}{L}\right).  
\end{alignat}

Then,
\begin{alignat}{1}
\phi_\theta &= 1 - \frac{1}{1 + \left(\frac{N}{L}\right)^2} \frac{L N_\theta - N L_\theta}{L^2} \\
&= 1 - \frac{LN_\theta + N^2}{L^2 + N^2}\\
&= \frac{L^2-LN_\theta}{L^2+N^2} \\
&= \frac{L(L-N_\theta)}{L^2+N^2} \\
&= \frac{L p_\theta}{\rho N (N^2+L^2)}.
\end{alignat}
However, recall that in a forward wave
\begin{alignat}{1}
\sgn(p_\theta)=\sgn(L),
\end{alignat}
and so $\phi_\theta$ is positive.  Therefore, as the gas particles travel through the shock, $\theta$ \emph{decreases}, and so $\angle(u,v)$ \emph{decreases} as well.  \\

For backward waves,

\begin{alignat}{1}
\phi = \theta - \arctan\left(\frac{N}{L}\right) + \pi,
\end{alignat}
giving the same expression for $\phi_\theta$.  But in this case
\begin{alignat}{1}
\sgn(p_\theta) = -\sgn(L),
\end{alignat}
but $N = -c$ and so $\phi_\theta$ is still positive.  However, for backward waves the gas particles move in the increasing $\theta$ direction, and so the flow still turns towards the origin as the gas particles travel through the wave.  See Figure \myref{twosectors} for some examples of waves.\\

For shocks, since $L$ is continuous and $|N|$ decreases as the gas particles pass through the shocks, the flow is turned away from the origin if $L$ is positive, and toward the origin if $L$ is negative. \\

\begin{figure}[h] 
  \centering
  \def\svgwidth{200pt}
  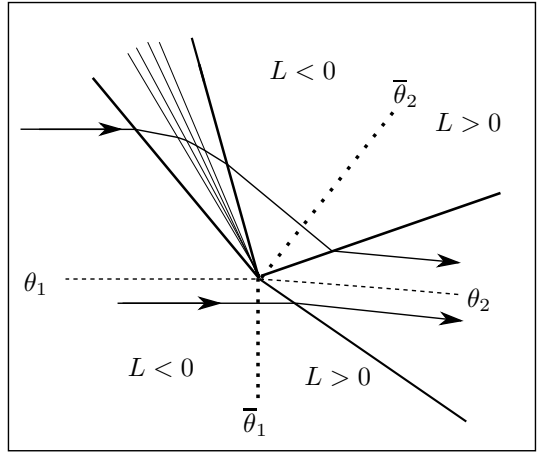
  \caption{An example with two sectors.  $I_1 = [\theta_1, \theta_2]$ is on the bottom and is a backward sector, $I_2 = [\theta_2, \theta_1]$ is on the top and is a forward sector.  In $I_1$, there is a single shock in the $L>0$ part.  In $I_2$, the gas passes through a shock, a compression wave, then a shock in the $L<0$ part, and a shock in the $L>0$ part.  As the gas particles travel through shocks, the flow is turned toward the shock line, and compression waves turn the flow toward the origin.  Rarefaction waves would turn the flow toward the origin as well.}
  \mylabel{twosectors}
\end{figure}

\begin{theorem}
For $\gamma = 1.4$, there can be a maximum of two sectors.  For other values of $\gamma > 1$, there can be up to three sectors, but there are no values of $\gamma > 1$ that lead to flows with more than three sectors. 
\end{theorem}
\begin{proof}
Choose coordinates so that there is a forward sector $I = [\alpha, \pi]$ where $0 < \alpha < \pi$, so that $N(\pi) = N(\alpha) = 0$.  Then, by Lemma \myref{tangprop} $L(\alpha+) > 0$, $L(\pi-) < 0$, and so the flow needs to be turned away from the origin by an angle of $\alpha$.  Following \myref{tangprop}, denote $\overline{\theta}$ the unique value between $\alpha$ and $\pi$ such that $L(\overline{\theta}) = 0$, and recall that $L$ is positive on $]\alpha, \overline{\theta}[$, and negative on $]\overline{\theta}, \pi[$.  \\

By Theorem \myref{onewave}, the discussion after it, and the discussion preceding this theorem, any shocks or compression waves on $]\overline{\theta}, \pi[$ turn the flow towards the origin.  Therefore, $\angle\big(u(\overline{\theta}),v(\overline{\theta})\big) \leq 0$.  Since we are interested in finding the maximum possible $\alpha$ that the flow can be turned upwards, the best possible situation is for there to be no compression waves or shocks on $]\overline{\theta},\pi[$, which yields $\overline{\theta} = \frac{\pi}{2}, \angle\big(u(\pi/2),v(\pi/2)\big) = 0$.\\

If there is a shock at $\frac{\pi}{2}$, then it is a normal shock and so the flow is absolutely subsonic, and thus constant, for $]\alpha, \frac{\pi}{2}[$.  Therefore $\alpha = 0$ since the flow can never be turned away from the origin. \\

Therefore, to accomplish the maximum upwards turning, the flow should be constant on $[\frac{\pi}{2}, \pi[$.  By Theorem \myref{Lpos}, the flow is either constant (again resulting in $\alpha = 0$), has exactly one rarefaction, or exactly one shock on $]\alpha, \frac{\pi}{2}[$.  A rarefaction wave turns the flow towards the origin, resulting in $\alpha < 0$, and so there must be a single shock to accomplish $\alpha > 0$.  \\

Using the well known $\theta-\beta-M$ equation (see \cite{krause}, Chapter 4) to relate the incident Mach number $M := \frac{|\vec{u}|_+}{c_+}$, the turning angle $\alpha$, and the shock angle $\theta$, we have that
\begin{alignat}{1}
\alpha = \arctan \left( \frac{2 \cot \theta (M^2 \sin^2\theta -1)}{M^2(\gamma+\cos(2 \theta))+2}\right).
\end{alignat}
It is well known (see \cite{courant-friedrichs}) that the curves $\alpha(\theta)$ for fixed values of $M$ all lie below the limiting case $M\rightarrow \infty$, and solving for the maximum $\alpha$ yields
\begin{alignat}{1}
\alpha_\textrm{max} = \arcsin \left( \frac{1}{\gamma} \right).
\end{alignat}

The flow can only be turned upward when $L>0$, so there can never be more than three sectors since $\alpha_\textrm{max} = \frac{\pi}{2}$ is only attained in the limit as $\gamma \searrow 1$.    For $\gamma = 1.4$, $\alpha_\textrm{max} \approx 45.5 ^\circ$, and so the flow cannot turn the required $60 ^\circ$ needed to have more than two sectors.  For $1 < \gamma < 1.15$, $\alpha_{\max} > 60 ^\circ$, and so there will exist finite incoming Mach numbers for which the flow can turn $60 ^\circ$, allowing for the existence flows with three sectors for some values of $\gamma >1$.
\end{proof}

\begin{figure}[h]
  \centering
  \def\svgwidth{200pt}
  \vspace{-1cm}
  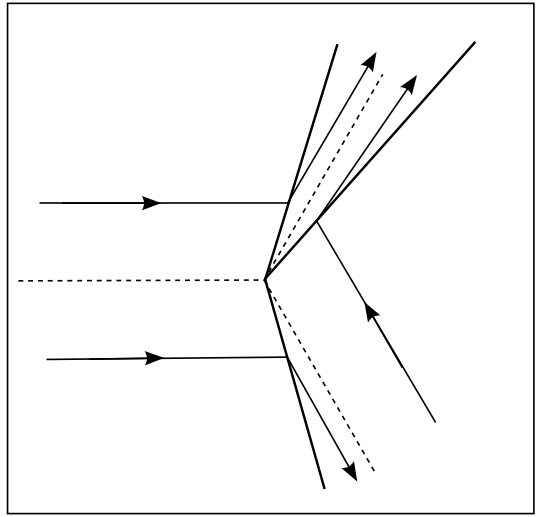
  \caption{For $\gamma < 1.15$, the maximum turning angle is greater than $60 ^\circ$.  Therefore there exist flows with three contact discontinuities, such as the one above.  In this example each sector has one shock in the region where $L > 0$, causing the flow to turn away from the origin.}
   \mylabel{threesectors}
\end{figure}

\newpage
 \providecommand{\bysame}{\leavevmode\hbox to3em{\hrulefill}\thinspace}
\providecommand{\MR}{\relax\ifhmode\unskip\space\fi MR }
\providecommand{\MRhref}[2]{%
  \href{http://www.ams.org/mathscinet-getitem?mr=#1}{#2}
}
\providecommand{\href}[2]{#2}

\bibliographystyle{amsalpha}

\end{document}